\documentclass[12pt]{article}
\usepackage{amssymb,amsmath,amsthm, amsfonts}
\usepackage{graphicx}
\usepackage{epsfig}
\usepackage{tikz}
\usepackage{mathrsfs}

\def\disp{\displaystyle}

\def\dref#1{(\ref{#1})}

\theoremstyle{plain}
\newtheorem{theorem}{Theorem}[section]
\newtheorem{lemma}{Lemma}[section]
\newtheorem{proposition}{Proposition}[section]

\theoremstyle{definition}
\newtheorem{definition}{Definition}[section]

\newtheorem{remark}{Remark}[section]

\numberwithin{equation}{section}

\begin{document}

\title{\bf An optimal  result  for global existence and boundedness in a
three-dimensional
 Keller-Segel-Stokes system   with  nonlinear diffusion
}

\author{
Jiashan Zheng\thanks{Corresponding author.   E-mail address:
 zhengjiashan2008@163.com (J.Zheng)}
 \\
    School of Mathematics and Statistics Science,\\
     Ludong University, Yantai 264025,  P.R.China \\
}
\date{}


\maketitle \vspace{0.3cm}
\noindent
\begin{abstract}
This paper investigates the following  Keller-Segel-Stokes system with  nonlinear diffusion
$$
 \left\{
 \begin{array}{l}
   n_t+u\cdot\nabla n=\Delta n^m-\nabla\cdot(n\nabla c),\quad
x\in \Omega, t>0,\\
    c_t+u\cdot\nabla c=\Delta c-c+n,\quad
x\in \Omega, t>0,\\
u_t+\nabla P=\Delta u+n\nabla \phi,\quad
x\in \Omega, t>0,\\
\nabla\cdot u=0,\quad
x\in \Omega, t>0\\
 \end{array}\right.\eqno(KSF)
 $$
 under homogeneous boundary conditions of Neumann type for $n$ and $c$, and of Dirichlet type for $u$ in
a  three-dimensional bounded domains $\Omega\subseteq \mathbb{R}^3$ with smooth boundary, where 
$\phi\in W^{1,\infty}(\Omega),m>0$.
It is proved that if $m>\frac{4}{3}$,
  then for any sufficiently regular nonnegative
initial data there exists at least one global boundedness  solution  for system  $(KSF)$, which in view of the known
results for the fluid-free system  mentioned below (see Introduction)  is an {\bf optimal} restriction on $m$.

\end{abstract}

\vspace{0.3cm}


\newpage
\section{Introduction}
Many phenomena, which appear in natural science, especially, biology and
physics, support animals' lives (see \cite{Liggh00,Xu5566r793,Wangsseeess21215,Guggg1215}).
Chemotaxis is the biological phenomenon of oriented movement of
cells under influence of a chemical signal substance (see Keller and Segel  \cite{Keller2710}).
A classical mathematical model for this type of processes was proposed by Keller and Segel in
\cite{Keller2710} as follows:
\begin{equation}
 \left\{\begin{array}{ll}
 n_t=\Delta n-\chi\nabla\cdot(n\nabla c),
 \quad
x\in \Omega,~ t>0,\\
 \disp{ c_t=\Delta c- c +n}\quad
x\in \Omega, ~t>0,\\
 \end{array}\right.\label{7223ffggddff44101.2x16677}
\end{equation}
where $\chi > 0$ is called chemotactic sensitivity, $n$ and $c$ denote the density of the cell population and the
concentration of the attracting chemical substance, respectively.
Starting from the pioneering work of Keller and Segel (see Keller and Segel \cite{Keller2710}),
an extensive mathematical literature has
grown on the Keller-Segel model and its variants (see e.g.  \cite{Bellomo1216,Hillen334,Horstmann791,Horstmann2710}).
To prevent any chemotactic collapse in \dref{7223ffggddff44101.2x16677}, the following
variant has also been widely investigated
\begin{equation}
 \left\{\begin{array}{ll}
 n_t=\Delta n^m-\chi\nabla\cdot( n\nabla c),\\
 \disp{ c_t=\Delta c- c +n.}
 \end{array}\right.\label{7223ddfffghhf44101.2x16677}
\end{equation}
 The main issue of the investigation was whether the
solutions of the models are bounded or blow-up. In fact, 
all
solutions are global and uniformly bounded if $m>2-\frac{2}{N}$ (see Tao and Winkler \cite{Tao794,Winkler79}), whereas if $m<2-\frac{2}{N}$, \dref{7223ddfffghhf44101.2x16677} possess some solutions which blow up in finite time (see  Winkler et. al. \cite{Cie791,Winkler79}).
 Therefore,
\begin{equation}m=2-\frac{2}{N}
\label{722dff344101.ddff2ffddfffggffggx16677}
\end{equation}
is the critical blow-up exponent, which is related to the presence of a so-called volume-filling effect.
For a more detailed discussion on this issue and on a parabolic-elliptic version of Keller-Segel system and its variants
we refer readers to see  Winkler et al. \cite{Cie72,Winkler72,Winkler21215,Winkler79,Wangffgggss21215}, Zheng et al. \cite{Zhengssssdefr23,Zheng00,Zhengssdefr23,Zheng33312186,Zhengssdddssddddkkllssssssssdefr23,Zhengssssssdefr23}.

In various situations, however, the migration of bacteria is furthermore substantially affected by
changes in their environment (see \cite{Miller7gg6,Tuval1215}).
Since the bacteria
consume the chemical instead of producing it,
Tuval et al. (\cite{Tuval1215}) proposed the following (quasilinear) chemotaxis(-Navier)-Stokes system
 \begin{equation}
 \left\{\begin{array}{ll}
   n_t+u\cdot\nabla n=\Delta n^m-\nabla\cdot( nS(x,n,c)\nabla c),\quad
x\in \Omega, t>0,\\
    c_t+u\cdot\nabla c=\Delta c-nf(c),\quad
x\in \Omega, t>0,\\
u_t+\kappa (u\cdot\nabla)u+\nabla P=\Delta u+n\nabla \phi,\quad
x\in \Omega, t>0,\\
\nabla\cdot u=0,\quad
x\in \Omega, t>0\\
 \end{array}\right.\label{1.1hhjjddssggtyy}
\end{equation}
in a bounded domain
 $\Omega \subset \mathbb{R}^3$ with smooth boundary, where $f(c)$ is the consumption rate of the oxygen by the cells and
 $S(x, n, c)$ is a tensor-valued function or a scalar function which satisfies
 \begin{equation}\label{x1.73142vghf48gg}|S(x, n, c)|\leq C_S(1 + n)^{-\alpha} ~~~~\mbox{for all}~~ (x, n, c)\in\Omega\times [0,\infty)^2
 \end{equation}
with some $C_S > 0$ and $\alpha> 0$.
Here  
$n$ and $c$ are defined as before, 
 $u,P,\phi$ and $\kappa\in \mathbb{R}$ denote, respectively, the velocity field, the associated pressure of the fluid, the potential of the
 gravitational field and the
strength of nonlinear fluid convection. 
In recent years, approaches have been developed based on a natural energy functional, in the past several years
there have been numerous analytical approaches that addressed issues of the solvability result for system \dref{1.1hhjjddssggtyy} with  $S(x, n, c):=S(n)$ is a scalar function (see e.g.
Chae et al.  \cite{Chaexdd12176},
Duan et al. \cite{Duan12186},
Liu and Lorz  \cite{Liu1215,Lorz1215},
 Tao and Winkler   \cite{Tao41215,Winkler31215,Winkler61215,Winkler51215}, Zhang et. al.  \cite{Zhang12176,Zhangdddddff4556} and references therein).
 On the other hand,
for general $S$ is  a chemotactic sensitivity tensor, 
one can see Winkler (\cite{Winkler11215}) and Zheng (\cite{Zhengsddfff00})
 and the references therein for details.

 Concerning the framework where the chemical is produced by the cells instead of consumed, then corresponding  chemotaxis--fluid model
 is then the quasilinear  Keller-Segel-Stokes system of the form
 \begin{equation}
 \left\{\begin{array}{ll}
   n_t+u\cdot\nabla n=\Delta n^m-\nabla\cdot(nS(x,n,c)\nabla c),\quad
x\in \Omega, t>0,\\
    c_t+u\cdot\nabla c=\Delta c-c+n,\quad
x\in \Omega, t>0,\\
u_t+\kappa (u\cdot\nabla)u+\nabla P=\Delta u+n\nabla \phi,\quad
x\in \Omega, t>0,\\
\nabla\cdot u=0,\quad
x\in \Omega, t>0,\\
 \end{array}\right.\label{1ssxdcfvgb.1}
\end{equation}
 which
  describe
chemotaxis-fluid interaction in cases when the evolution of the chemoattractant is essentially
dominated by production through cells (\cite{Bellomo1216,Hillen}).
%

Compared to the classical Keller-Segel chemotaxis and chemotaxis(-Navier)-Stokes system \dref{1.1hhjjddssggtyy}, the mathematical analysis of Keller-Segel-Stokes system \dref{1ssxdcfvgb.1} has to cope with considerable additional challenges.
 (see Wang, Xiang et. al.  \cite{Peng55667,Wang21215,Wangss21215,Wangssddss21215}, Zheng \cite{Zhengsddfff00,Zhengssdddd00}).
Xiang et. al. (\cite{Liggghh793}) established the global existence and boundedness of
the 2D system \dref{1ssxdcfvgb.1} under the assumption of \dref{x1.73142vghf48gg} with $\alpha = 0$ for any $m > 1$.
In a three-dimensional setup involving linear diffusion ($m = 1$ in \dref{1ssxdcfvgb.1}) and tensor-valued sensitivity
$S$ satisfying \dref{x1.73142vghf48gg} global weak solutions have been shown to exists
for $\alpha > \frac{3}{7}$ (see \cite{LiuZhLiuLiuandddgddff4556}) and $\alpha > \frac{1}{3}$  (see \cite{Zhengssdddd00} and also \cite{Wangssddss21215}), respectively.
If the bacteria
diffuses in a porous medium ($m\neq1$) and sensitivity $S(x, n, c)\equiv 1$, the global weak solutions for \dref{1ssxdcfvgb.1} whenever $m > 2$ (\cite{Zhengsddfffsdddssddddkkllssssssssdefr23}), which most
probably is not optimal in the sense of $m+ \alpha > 2-\frac{2}{N}$ (\cite{Winkler79}).
For the more related works in this direction, we mention that a corresponding quasilinear
version or the logistic damping
has been deeply investigated by  Zheng
\cite{Zhengsddfff00,Zhengsdsd6}, Wang and Liu \cite{Liuddfffff}, Tao and Winkler \cite{Tao41215},
 Wang et. al.
\cite{Wang21215,Wangss21215}.

Motivated by the above works,
  the aim of the present paper is to
   study the following Keller-Segel-Stokes system   with  nonlinear diffusion
 \begin{equation}
 \left\{\begin{array}{ll}
   n_t+u\cdot\nabla n=\Delta n^m-\nabla\cdot(n\nabla c),\quad
x\in \Omega, t>0,\\
    c_t+u\cdot\nabla c=\Delta c-c+n,\quad
x\in \Omega, t>0,\\
u_t+\nabla P=\Delta u+n\nabla \phi,\quad
x\in \Omega, t>0,\\
\nabla\cdot u=0,\quad
x\in \Omega, t>0,\\
 \disp{\nabla n\cdot\nu=\nabla c\cdot\nu=0,u=0,}\quad
x\in \partial\Omega, t>0,\\
\disp{n(x,0)=n_0(x),c(x,0)=c_0(x),u(x,0)=u_0(x),}\quad
x\in \Omega.\\
 \end{array}\right.\label{1.1}
\end{equation}

This paper is organized as follows. In Section 2, we state the main results, give an approximate problem and some basic
properties.
%
%
In Section 3, we derive an upper bound for regularized problems of \dref{1.1} by using the Maximal Sobolev regularity and a Moser-type iteration.
Finally, in Section 4 we
prove our main results  by passage to the limit in the approximate problem via estimates from
Section 3.

\section{Preliminaries and  main results}

In this section, we give some notations and recall some basic facts which will be
frequently used throughout the paper.
To formulate the main result, let us suppose that
\begin{equation}
\phi\in W^{1,\infty}(\Omega)
\label{dd1.1fghyuisdakkkllljjjkk}
\end{equation}
 and the initial data
$(n_0, c_0, u_0)$ fulfills
\begin{equation}\label{ccvvx1.731426677gg}
\left\{
\begin{array}{ll}
\displaystyle{n_0\in C^\kappa(\bar{\Omega})~~\mbox{for certain}~~ \kappa > 0~~ \mbox{with}~~ n_0\geq0 ~~\mbox{in}~~\Omega},\\
\displaystyle{c_0\in W^{2,\infty}(\Omega)~~\mbox{with}~~c_0\geq0~~\mbox{in}~~\bar{\Omega},}\\
\displaystyle{u_0\in D(A^\gamma_{r})~~\mbox{for~~ some}~~\gamma\in ( \frac{3}{4}, 1)~~\mbox{and any}~~ {r}\in (1,\infty),}\\
\end{array}
\right.
\end{equation}
where $A_{r}$ denotes the Stokes operator with domain $D(A_{r}) := W^{2,{r}}(\Omega)\cap  W^{1,{r}}_0(\Omega)
\cap L^{r}_{\sigma}(\Omega)$,
and
$L^{r}_{\sigma}(\Omega) := \{\varphi\in  L^{r}(\Omega)|\nabla\cdot\varphi = 0\}$ for ${r}\in(1,\infty)$
 (\cite{Sohr}).

\begin{theorem}\label{theorem3}
Let $\Omega \subset \mathbb{R}^3$ be a bounded  domain with smooth boundary, with smooth boundary.
Suppose that the assumptions
 \dref{dd1.1fghyuisdakkkllljjjkk} and \dref{ccvvx1.731426677gg}
 hold.
 If 
\begin{equation}\label{x1.73142vghf48}m>\frac{4}{3},
\end{equation}
then 
 the problem \dref{1.1} possesses at least
one global weak solution $(n, c, u, P)$
 in the sense of Definition \ref{df1}. Moreover, this solution is bounded in
$\Omega\times(0,\infty)$ in the sense that
\begin{equation}
\|n(\cdot, t)\|_{L^\infty(\Omega)}+\|c(\cdot, t)\|_{W^{1,\infty}(\Omega)}+\| u(\cdot, t)\|_{W^{1,\infty}(\Omega)}\leq C~~ \mbox{for all}~~ t>0.
\label{1.163072xggttyyu}
\end{equation}
Furthermore, $c$ and $u$ are continuous in
$\bar{\Omega}\times[0,\infty)$ and
\begin{equation}
n\in C^0_{\omega-*}([0,\infty); L^\infty(\Omega)).
 \label{zjscz2.5297x96302222tt4455hyuhiigg}
 \end{equation}
\end{theorem}
\begin{remark}
(i) From Theorem \ref{theorem3}, we conclude that  if
the exponent $m$ of nonlinear diffusion  is large than $\frac{4}{3}$, then
 model \dref{1.1} exists a global (weak) bounded solution,   which yields to the nonlinear diffusion term  benefits the global of solutions.

(ii) In comparison to the result
for the corresponding fluid-free system \cite{Cie201712791,Tao794,Winkler79}, it is easy to see that the restriction on $m$ here is
{\bf optimal}.


(iii) Obviously, $ 2>\frac{4}{3}$,  Theorem \ref{theorem3} seems to partly improve the results of Zheng (\cite{Zhengsddfffsdddssddddkkllssssssssdefr23}),  who showed the global weak existence of solutions for \dref{1ssxdcfvgb.1} in
 the cases $S(x, n, c)\equiv1$  with $m> 2$.

(iv) If $\alpha=0,$ then  $ \max\{2-2\alpha,\frac{3}{4}\}=2>\frac{4}{3}$, so that,  Theorem \ref{theorem3} also (partly) improve the results of Peng and  Xiang (\cite{Peng55667}),  who showed the global weak existence of solutions for \dref{1ssxdcfvgb.1} in
 the cases $S(x, n, c)$ satisfying \dref{x1.73142vghf48gg} with $m> \max\{2-2\alpha,\frac{3}{4}\}$.

 (v) We should pointed that the idea of this paper can be also solved with other types of models, e.g.
 an attraction-repulsion chemotaxis fluid
model with nonlinear diffusion  (see \cite{Zhengssdddddfssdddd00}).
\end{remark}

In order to construct solutions of \dref{1.1} through an appropriate approximation, we then need to consider
the approximate system
\begin{equation}
 \left\{\begin{array}{ll}
   n_{\varepsilon t}+u_{\varepsilon}\cdot\nabla n_{\varepsilon}=\Delta (n_{\varepsilon}+\varepsilon)^m-\nabla\cdot(n_{\varepsilon}\nabla c_{\varepsilon}),\quad
x\in \Omega, t>0,\\
    c_{\varepsilon t}+u_{\varepsilon}\cdot\nabla c_{\varepsilon}=\Delta c_{\varepsilon}-c_{\varepsilon}+n_{\varepsilon},\quad
x\in \Omega, t>0,\\
u_{\varepsilon t}+\nabla P_{\varepsilon}=\Delta u_{\varepsilon}+n_{\varepsilon}\nabla \phi,\quad
x\in \Omega, t>0,\\
\nabla\cdot u_{\varepsilon}=0,\quad
x\in \Omega, t>0,\\
 \disp{\nabla n_{\varepsilon}\cdot\nu=\nabla c_{\varepsilon}\cdot\nu=0,u_{\varepsilon}=0,\quad
x\in \partial\Omega, t>0,}\\
\disp{n_{\varepsilon}(x,0)=n_0(x),c_{\varepsilon}(x,0)=c_0(x),u_{\varepsilon}(x,0)=u_0(x)},\quad
x\in \Omega.\\
 \end{array}\right.\label{1.1fghyuisda}
\end{equation}


Next, we will provide some results which will be used later. To this end,  by an adaptation of well-established fixed point arguments, one can readily verify local existence theory for \dref{1.1fghyuisda}
(see \cite{Winkler11215}, Lemma 2.1 of \cite{Painter55677} and Lemma 2.1 of \cite{Winkler51215}).
%
%
\begin{lemma}\label{lemma70}
Assume
that $\varepsilon\in(0,1).$
%
Then there exist $T_{max}\in  (0,\infty]$ and
a classical solution $(n_\varepsilon, c_\varepsilon, u_\varepsilon, P_\varepsilon)$ of \dref{1.1fghyuisda} in
$\Omega\times(0, T_{max})$ such that
\begin{equation}
 \left\{\begin{array}{ll}
 n_\varepsilon\in C^0(\bar{\Omega}\times[0,T_{max}))\cap C^{2,1}(\bar{\Omega}\times(0,T_{max})),\\
  c_\varepsilon\in  C^0(\bar{\Omega}\times[0,T_{max}))\cap C^{2,1}(\bar{\Omega}\times(0,T_{max})),\\
  u_\varepsilon\in  C^0(\bar{\Omega}\times[0,T_{max}))\cap C^{2,1}(\bar{\Omega}\times(0,T_{max})),\\
  P_\varepsilon\in  C^{1,0}(\bar{\Omega}\times(0,T_{max})),\\
   \end{array}\right.\label{1.1ddfghyuisda}
\end{equation}
 classically solving \dref{1.1fghyuisda} in $\Omega\times[0,T_{max})$.
%
Moreover,  $n_\varepsilon$ and $c_\varepsilon$ are nonnegative in
$\Omega\times(0, T_{max})$, and
\begin{equation}
\|n_\varepsilon(\cdot, t)\|_{L^\infty(\Omega)}+\|c_\varepsilon(\cdot, t)\|_{W^{1,\infty}(\Omega)}+\|A^\gamma u_\varepsilon(\cdot, t)\|_{L^{2}(\Omega)}\rightarrow\infty~~ \mbox{as}~~ t\rightarrow T_{max},
\label{1.163072x}
\end{equation}
where $\gamma$ is given by \dref{ccvvx1.731426677gg}.
\end{lemma}
Given all $s_0\in (0, T_{max})$, from the regularity properties asserted by Lemma \ref{lemma70} we know that
 there exists
$\beta>0$ such that
\begin{equation}\label{eqx45xx1ddfgggg2112}
\|n_\varepsilon(\tau)\|_{L^\infty(\Omega)}\leq \beta~~~\|u_\varepsilon(\tau)\|_{W^{1,\infty}(\Omega)}\leq \beta~~\mbox{and}~~\|c_\varepsilon(\tau)\|_{W^{2,\infty}(\Omega)}\leq \beta~~\mbox{for all}~~\tau\in[0,s_0].
\end{equation}

\begin{lemma}(\cite{Winkler11215})\label{lemma630jklhhjj}
 Let $l\in[1,+\infty)$ and $r\in[1,+\infty]$ be such that
 \begin{equation}
\left\{\begin{array}{ll}
l<\frac{3r}{3-r}~~\mbox{if}~~
r\leq 3,\\
l\leq\infty~~\mbox{if}~~
r>3.
 \end{array}\right.\label{3.10gghhjuulooll}
\end{equation}
Then for all $K > 0$ there exists $C = C(l, r,K)$ such that 
 \begin{equation}\|n(\cdot, t)\|_{L^r(\Omega)}\leq K~~ \mbox{for all}~~ t\in(0, T_{max}),
\label{3.10gghhjuuloollgghhhy}
\end{equation}
then
 \begin{equation}\|D u(\cdot, t)\|_{L^l(\Omega)}\leq C~~ \mbox{for all}~~ t\in(0, T_{max}).
\label{3.10gghhjuuloollgghhhyhh}
\end{equation}
\end{lemma}
\begin{lemma}\label{lemma45xy1222232}(\cite{Hieber,Zhengddkkllssssssssdefr23,Zhengssdddssddddkkllssssssssdefr23})
Suppose  $\gamma\in (1,+\infty)$, $g\in L^\gamma((0, T); L^\gamma(
\Omega))$ and  $v_0\in W^{2,\gamma}(\Omega)$
such that $\disp\frac{\partial v_0}{\partial \nu}=0$.
Let $v$ be a solution of the following initial boundary value
 \begin{equation}
 \left\{\begin{array}{ll}
 v_t -\Delta v=g,~~(x, t)\in
\Omega\times(0, T ),\\
\disp\frac{\partial v}{\partial \nu}=0,~~(x, t)\in
 \partial\Omega\times(0, T ),\\
v(x,0)=v_0(x),~~~(x, t)\in
 \Omega.\\
 \end{array}\right.\label{33331.3xcx29}
\end{equation}
Then there exists a positive constant $\delta_0$ such that
\begin{equation}
\begin{array}{rl}
&\disp{\int_0^T\|v(\cdot,t)\|^{\gamma}_{L^{\gamma}(\Omega)}dt+\int_0^T\|v_t(\cdot,t)\|^{\gamma}_{L^{\gamma}(\Omega)}dt+\int_0^T\|\Delta v(\cdot,t)\|^{\gamma}_{L^{\gamma}(\Omega)}dt}\\
\leq&\disp{\delta_0\left(\int_0^T\|g(\cdot,t)\|^{\gamma}_{L^{\gamma}(\Omega)}dt+\|v_0(\cdot,t)\|^{\gamma}_{L^{\gamma}(\Omega)}+\|\Delta v_0(\cdot,t)\|^{\gamma}_{L^{\gamma}(\Omega)}\right).}\\
\end{array}
\label{cz2.5bbv}
\end{equation}
On the other hand,
assuming $v$ is
a solution of the following initial boundary value
 \begin{equation}
 \left\{\begin{array}{ll}
 v_t -\Delta v+v=g,~~~(x, t)\in
 \Omega\times(0, T ),\\
\disp\frac{\partial v}{\partial \nu}=0,~~~(x, t)\in
 \partial\Omega\times(0, T ),\\
v(x,0)=v_0(x),~~~(x, t)\in
 \Omega.\\
 \end{array}\right.\label{33331.3xcx29}
\end{equation}
Then there exists a positive constant $C_{\gamma}:=C_{\gamma,|\Omega|}$ such that if $s_0\in[0,T)$, $v(\cdot,s_0)\in W^{2,\gamma}(\Omega)(\gamma>N)$ with $\disp\frac{\partial v(\cdot,s_0)}{\partial \nu}=0,$ then
\begin{equation}
\begin{array}{rl}
&\disp{\int_{s_0}^Te^{\gamma  s}(\| v(\cdot,t)\|^{\gamma}_{L^{\gamma}(\Omega)}+\|\Delta v(\cdot,t)\|^{\gamma}_{L^{\gamma}(\Omega)})ds}\\
\leq &\disp{C_{\gamma}\left(\int_{s_0}^Te^{\gamma  s}
\|g(\cdot,s)\|^{\gamma}_{L^{\gamma}(\Omega)}ds+e^{\gamma  s}(\|v_0(\cdot,s_0)\|^{\gamma}_{L^{\gamma}(\Omega)}+\|\Delta v_0(\cdot,s_0)\|^{\gamma}_{L^{\gamma}(\Omega)})\right).}\\
\end{array}
\label{3333cz2.5bbv114}
\end{equation}
\end{lemma}

\section{A priori estimates}
In this section, we proceed to derive $\varepsilon$-independent estimates for the approximate solutions constructed above. The iteration depends on a series of a priori estimate. To this end, throughout
this section, for any $\varepsilon\in(0, 1)$, we let $(n_\varepsilon, c_\varepsilon, u_\varepsilon, P_\varepsilon)$  be the global solution of problem  \dref{1.1fghyuisda}.
The following estimates of $n_{\varepsilon}$ and $c_{\varepsilon}$ are basic but important in the proof of our result.
\begin{lemma}\label{fvfgfflemma45}
There exists 
$\lambda > 0$ independent of $\varepsilon$ such that the solution of \dref{1.1fghyuisda} satisfies
%
%
\begin{equation}
\int_{\Omega}{n_{\varepsilon}}+\int_{\Omega}{c_{\varepsilon}}\leq \lambda~~\mbox{for all}~~ t\in(0, T_{max}).
\label{ddfgczhhhh2.5ghju48cfg924ghyuji}
\end{equation}
%
%
\end{lemma}

We next show the following lemma which holds a key for the proof of Theorem \ref{theorem3}.
Employing the same  arguments as in the proof of Lemma 3.3 in \cite{Zhengssdddd00} (see also \cite{Zhengsddfff00}), we derive  the following Lemma: 
\begin{lemma}\label{lemmaghjssddgghhmk4563025xxhjklojjkkk}
Let $m>\frac{4}{3}$.
Then there exists $C>0$ independent of $\varepsilon$ such that the solution of \dref{1.1fghyuisda} satisfies
\begin{equation}
\begin{array}{rl}
&\disp{\int_{\Omega}(n_{\varepsilon}+\varepsilon)^{m-1}+\int_{\Omega}   c_{\varepsilon}^2+\int_{\Omega}  | {u_{\varepsilon}}|^2\leq C~~~\mbox{for all}~~ t\in (0, T_{max}).}\\
\end{array}
\label{czfvgb2.5ghhjuyuccvviihjj}
\end{equation}
Moreover, for $T\in(0, T_{max})$, it holds that
one can find a constant $C > 0$ independent of $\varepsilon$ such that
\begin{equation}
\begin{array}{rl}
&\disp{\int_{0}^T\int_{\Omega} \left[ (n_{\varepsilon}+\varepsilon)^{2m-4} |\nabla {n_{\varepsilon}}|^2+ |\nabla {c_{\varepsilon}}|^2+ |\nabla {u_{\varepsilon}}|^2\right]\leq C.}\\
\end{array}
\label{bnmbncz2.5ghhjuyuivvbnnihjj}
\end{equation}
\end{lemma}
In the following, we shall derive an important inequality, which plays a key role
in the proof of our main result.
\begin{lemma}\label{lemma45630223116}
Let
 $p=\frac{25}{16}$ 
  and $\theta=\frac{8}{7}$.
Then there exists a positive constant $\tilde{l}_0\in(\frac{1737}{582},3)$ such that
\begin{equation}
\frac{\frac{5}{6}-\frac{1}{\theta'(p+1)}}{\frac{7}{6}-\frac{1}{p+1}}+\frac{\frac{1}{\tilde{l}_0}-\frac{1}{\theta (p+1)}}{\frac{1}{\tilde{l}_0}+\frac{2}{3}-\frac{1}{p+1}}<1,
\label{zjscz2.5297x9630222211444125}
\end{equation}
where $\theta'=\frac{\theta}{\theta-1}=8.$
\end{lemma}
\begin{proof}
It can the readily be verified that 
 $$\frac{1}{p+1}=\frac{16}{41}$$
and
$$1-\frac{\frac{1}{\tilde{l}_0}-\frac{1}{\theta (p+1)}}{\frac{1}{\tilde{l}_0}+\frac{2}{3}-\frac{1}{p+1}}=\frac{\frac{2}{3}-\frac{1}{\theta' (p+1)}}{\frac{1}{\tilde{l}_0}+\frac{2}{3}-\frac{1}{p+1}},$$
due to our assumption $p=\frac{25}{16}$ and $\theta=\frac{8}{7}$.
These together with some basic calculation yield to \dref{zjscz2.5297x9630222211444125}.
\end{proof}

The following estimates are crucial to prove our main results, which are based on the
Maximal Sobolev regularity (see Hieber and Pr\"{u}ss \cite{Hieber}).
%
\begin{lemma}\label{lemma45566645630223}
If
  \begin{equation}\label{gddffffnjjmmx1.731426677gg}
m> \frac{4}{3},
\begin{array}{ll}\\
 \end{array}
\end{equation}
then there exists a positive constant $p_0>\frac{3}{2}$   independent of $\varepsilon$
such that 
 the solution of \dref{1.1} from Lemma \ref{lemma70} satisfies
\begin{equation}
\int_{\Omega}n^{p_0}_\varepsilon(x,t)dx\leq C ~~~\mbox{for all}~~ t\in(0,T_{max}).
\label{334444zjscz2.5297x96302222114}
\end{equation}
\end{lemma}
\begin{proof}
Let
$p=\frac{25}{16}$.
%
Multiplying the first equation of \dref{1.1} by ${(n_{\varepsilon}+\varepsilon)^{p-1}}$ and using $\nabla\cdot u_\varepsilon=0$, we derive that

\begin{equation}
\begin{array}{rl}
&\disp{\frac{1}{{p}}\frac{d}{dt}\|n_\varepsilon+\varepsilon\|^{{p}}_{L^{{p}}(\Omega)}+m({{p}-1})\int_{\Omega}(n_\varepsilon+\varepsilon)^{{m+{p}-3}}|\nabla n_\varepsilon|^2}
\\
=&\disp{-\int_\Omega (n_\varepsilon+\varepsilon)^{p-1}\nabla\cdot(n_\varepsilon
\nabla c_\varepsilon) }\\
=&\disp{(p-1)\int_\Omega  (n_\varepsilon+\varepsilon)^{p-2} n_\varepsilon
\nabla n_\varepsilon\cdot\nabla c_\varepsilon~~\mbox{for all}~~ t\in(0,T_{max}),}\\
\end{array}
\label{3333cz2.5114114}
\end{equation}
which yields to 
\begin{equation}
\begin{array}{rl}
&\disp{\frac{1}{{p}}\frac{d}{dt}\|n_\varepsilon+\varepsilon\|^{{{p}}}_{L^{{p}}(\Omega)}+m({{p}-1})\int_{\Omega}(n_\varepsilon+\varepsilon)^{{m+{p}-3}}|\nabla n_\varepsilon|^2}
\\
\leq&\disp{-\frac{{p}+1}{{p}}\int_{\Omega} (n_\varepsilon+\varepsilon)^{p} +(p-1)\int_\Omega  (n_\varepsilon+\varepsilon)^{p-2} n_\varepsilon
\nabla n_\varepsilon\cdot\nabla c_\varepsilon
   + \frac{{p}+1}{{p}}\int_\Omega (n_\varepsilon+\varepsilon)^{p}}\\
\end{array}
\label{3333cz2.5kk1214114114}
\end{equation}
for all $t\in(0,T_{max}).$

Here, according to  the Young inequality,  it reads that
%
%
\begin{equation}
\begin{array}{rl}
&\disp{\frac{{p}+1}{{p}}\int_\Omega   (n_\varepsilon+\varepsilon)^p \leq\varepsilon_1\int_\Omega  n^{{{p}+m-\frac{1}{3}}} +C_1(\varepsilon_1,p),}
\end{array}
\label{3333cz2.563011228ddff}
\end{equation}
where  
$$C_1(\varepsilon_1,{p})=\frac{m-\frac{1}{3}}{{p}+m-\frac{1}{3}}\left(\varepsilon_1\frac{{p}+m-\frac{1}{3}}{{p}}\right)^{-\frac{p}{m-\frac{1}{3}} }
\left(\frac{{p}+1}{{p}}\right)^{\frac{{p}+m-\frac{1}{3}}{m-\frac{1}{3}} }|\Omega|.$$
%
%

 Once more integrating by parts, combine  with \dref{x1.73142vghf48gg},  we also find
that
\begin{equation}
\begin{array}{rl}
&\disp{(p-1)\int_\Omega  (n_\varepsilon+\varepsilon)^{p-2} n_\varepsilon
\nabla n_\varepsilon\cdot\nabla c_\varepsilon }
\\
=&\disp{(p-1)\int_\Omega \nabla \int_0^{n_\varepsilon}(\tau+\varepsilon)^{p-2}  \tau d\tau\cdot\nabla c_\varepsilon  }
\\
=&\disp{-(p-1)\int_\Omega \int_0^{n_\varepsilon}(\tau+\varepsilon)^{p-2}  \tau d\tau \Delta c_\varepsilon }
\\
\leq&\disp{\frac{({{p}-1})}{p}\int_\Omega  (n_\varepsilon+\varepsilon)^{p}|\Delta c_\varepsilon| .}
\\
\end{array}
\label{3333c334444z2.563019114}
\end{equation}
Utilizing the Young inequality to the term on the right side of \dref{3333c334444z2.563019114}
leads to
\begin{equation}
\begin{array}{rl}
&\disp{\frac{({{p}-1})}{p}\int_\Omega  (n_\varepsilon+\varepsilon)^{p}|\Delta c_\varepsilon|}
\\
\leq&\disp{\int_\Omega  (n_\varepsilon+\varepsilon)^{{p}+1}+\frac{1}{ { {p}+1}}\left[\frac{ { {p}+1}}{ p}\right]^{- p}\left(\frac{({{p}-1})}{p} \right)^{ { {p}+1}}\int_\Omega |\Delta c_\varepsilon|^{ { {p}+1}} }
\\
=&\disp{\int_\Omega  (n_\varepsilon+\varepsilon)^{{p}+1}+{A}_1\int_\Omega |\Delta c_\varepsilon|^{ { {p}+1}} ,}
\\
\end{array}
\label{3333cz2.563019114gghh}
\end{equation}
where
$$A_1:=\frac{1}{ { {p}+1}}\left[\frac{ { {p}+1}}{ p}\right]^{- p}\left(\frac{({{p}-1})}{p} \right)^{ { {p}+1}}.$$
Hence \dref{3333cz2.5kk1214114114}, \dref{3333cz2.563011228ddff} and \dref{3333cz2.563019114gghh} 
results in
\begin{equation*}
\begin{array}{rl}
&\disp\frac{1}{{p}}\disp\frac{d}{dt}\|n_\varepsilon+\varepsilon\|^{{{p}}}_{L^{{p}}(\Omega)}+m(p-1)\int_{\Omega} (n_\varepsilon+\varepsilon)^{{m+{p}-3}}|\nabla n_\varepsilon|^2\\
\leq&\disp{\varepsilon_1\int_\Omega  (n_\varepsilon+\varepsilon)^{{{p}+m-\frac{1}{3}}} +\int_\Omega (n_\varepsilon+\varepsilon)^{{{p}+1}}-\frac{{p}+1}{{p}}\int_{\Omega} (n_\varepsilon+\varepsilon)^{p} }\\
&+\disp{{A}_1\int_\Omega |\Delta c_\varepsilon|^{ { {p}+1}} +
C_1(\varepsilon_1,{p})~~\mbox{for all}~~ t\in(0,T_{max}).}\\
\end{array}
\end{equation*}
Since, $m>\frac{4}{3}$, yields to ${{p}+1}<{{p}+m-\frac{1}{3}},$ therefore, by the Young inequality, we conclude that
\begin{equation}\label{3223444333c334444z2.563019114}
\begin{array}{rl}
&\disp\frac{1}{{p}}\disp\frac{d}{dt}\|n_\varepsilon+\varepsilon\|^{{{p}}}_{L^{{p}}(\Omega)}+\frac{4m(p-1)}{(m+p-1)^2}\|\nabla   (n_\varepsilon+\varepsilon)^{\frac{m+p-1}{2}}\|_{L^2(\Omega)}^{2}\\
\leq&\disp{2\varepsilon_1\int_\Omega  (n_\varepsilon+\varepsilon)^{{{p}+m-\frac{1}{3}}} -\frac{{p}+1}{{p}}\int_{\Omega} (n_\varepsilon+\varepsilon)^{p} }\\
&+\disp{{A}_1\int_\Omega |\Delta c_\varepsilon|^{ { {p}+1}} +
C_2(\varepsilon_1,{p}),}\\
\end{array}
\end{equation}
where
$$C_2(\varepsilon_1,{p})=\frac{m-\frac{4}{3}}{{p}+m-\frac{1}{3}}\left(\varepsilon_1\frac{{p}+m-\frac{1}{3}}{{p}+1}\right)^{-\frac{p+1}
{m-\frac{4}{3}} }
\left(\frac{{p}+1}{{p}}\right)^{\frac{{p}+m-\frac{1}{3}}{m-\frac{4}{3}} }|\Omega|.$$
On the other hand, by the Gagliardo--Nirenberg inequality and \dref{ddfgczhhhh2.5ghju48cfg924ghyuji}, one can get there exist positive constants  $\lambda_0$ and $\lambda_1$ such that
\begin{equation}
\begin{array}{rl}
&\disp\int_{\Omega}(n_\varepsilon+\varepsilon)^{p+m-\frac{1}{3}}\\
=&\disp{\| (n_\varepsilon+\varepsilon)^{\frac{m+p-1}{2}}\|^{\frac{2(p+m-\frac{1}{3})}{m+p-1}}_{L^{\frac{2(p+m-\frac{1}{3})}{m+p-1}}(\Omega)}}\\
\leq&\disp{\lambda_{0}(\|\nabla   (n_\varepsilon+\varepsilon)^{\frac{m+p-1}{2}}\|_{L^2(\Omega)}^{\frac{m+p-1}{p+m-\frac{1}{3}}}\|  (n_\varepsilon+\varepsilon)^{\frac{m+p-1}{2}}\|_{L^\frac{2}{m+p-1 }(\Omega)}^{1-\frac{m+p-1}{p+m-\frac{1}{3}}}+\|  (n_\varepsilon+\varepsilon)^{\frac{m+p-1}{2}}\|_{L^\frac{2}{m+p-1 }(\Omega)})^{\frac{2(p+m-\frac{1}{3})}{m+p-1}}}\\
\leq&\disp{\lambda_{1}(\|\nabla   (n_\varepsilon+\varepsilon)^{\frac{m+p-1}{2}}\|_{L^2(\Omega)}^{2}+1).}\\
\end{array}
\label{123cz2.57151hhddfffkkhhhjddffffgukildrftjj}
\end{equation}
In combination with \dref{3223444333c334444z2.563019114} and \dref{123cz2.57151hhddfffkkhhhjddffffgukildrftjj}, this shows that
\begin{equation}\label{3223444333c334444zsddfff2.563019114}
\begin{array}{rl}
\disp\frac{1}{{p}}\disp\frac{d}{dt}\|n_\varepsilon+\varepsilon\|^{{{p}}}_{L^{{p}}(\Omega)}\leq&\disp{(2\varepsilon_1-\frac{4m(p-1)}{(m+p-1)^2}\frac{1}{\lambda_1})\int_\Omega  (n_\varepsilon+\varepsilon)^{{{p}+m-\frac{1}{3}}} -\frac{{p}+1}{{p}}\int_{\Omega} (n_\varepsilon+\varepsilon)^{p} }\\
&+\disp{{A}_1\int_\Omega |\Delta c_\varepsilon|^{ { {p}+1}} +
C_3(\varepsilon_1,{p})~~\mbox{for all}~~ t\in(0,T_{max}),}\\
\end{array}
\end{equation}
where $$C_3(\varepsilon_1,{p})=C_2(\varepsilon_1,{p})+\frac{4m(p-1)}{(m+p-1)^2}.$$
For any $t\in (s_0,T_{max})$,
employing the variation-of-constants formula to \dref{3223444333c334444zsddfff2.563019114} and using $\varepsilon<1$, we obtain
\begin{equation}
\begin{array}{rl}
&\disp{\frac{1}{{p}}\|n_\varepsilon(t)+\varepsilon \|^{{{p}}}_{L^{{p}}(\Omega)}}
\\
\leq&\disp{\frac{1}{{p}}e^{-({p}+1)(t-s_0)}\|n_\varepsilon(s_0)+1 \|^{{{p}}}_{L^{{p}}(\Omega)}+(2\varepsilon_1-\frac{4m(p-1)}{(m+p-1)^2}\frac{1}{\lambda_1})\int_{s_0}^t
e^{-({p}+1)(t-s)}\int_\Omega (n_\varepsilon+\varepsilon)^{{{p}+m-\frac{1}{3}}} ds}\\
&+\disp{{A}_1\int_{s_0}^t
e^{-({p}+1)(t-s)}\int_\Omega |\Delta c_\varepsilon|^{ {p}+1} dxds+ C_3(\varepsilon_1,{p})\int_{s_0}^t
e^{-({p}+1)(t-s)}ds}\\
\leq&\disp{(2\varepsilon_1-\frac{4m(p-1)}{(m+p-1)^2}\frac{1}{\lambda_1}   )\int_{s_0}^t
e^{-({p}+1)(t-s)}\int_\Omega (n_\varepsilon+\varepsilon)^{{{p}+m-\frac{1}{3}}} ds+{A}_1\int_{s_0}^t
e^{-({p}+1)(t-s)}\int_\Omega |\Delta c_\varepsilon|^{ {p}+1} dxds}\\
&+\disp{C_4(\varepsilon_1,{p})}\\
\end{array}
\label{3333cz2.5kk1214114114rrgg}
\end{equation}
with
$$
\begin{array}{rl}
C_4:=C_4(\varepsilon_1,{p})=&\disp\frac{1}{{p}}e^{-({p}+1)(t-s_0)}\|n_\varepsilon(s_0)+1 \|^{{{p}}}_{L^{{p}}(\Omega)}+
 C_3(\varepsilon_1,{p})\int_{s_0}^t
e^{-({p}+1)(t-s)}ds.\\
\end{array}
$$
Due to \dref{ddfgczhhhh2.5ghju48cfg924ghyuji}, in view of Lemma \ref{lemma630jklhhjj}, we derive that  
\begin{equation}\|  Du_\varepsilon(\cdot, t)\|_{L^l(\Omega)}\leq C_5~~ \mbox{for all}~~ t\in(0, T_{max})~~ \mbox{and for any}~~~l<\frac{3}{2}.
\label{3.10gghhjukklllkklllooppuloollgghhhyhh}
\end{equation}
Here employing the three-dimensional Sobolev inequality, we can find
\begin{equation}\|  u_\varepsilon(\cdot, t)\|_{L^{l_0}(\Omega)}\leq C_6~~ \mbox{for all}~~ t\in(0, T_{max})~~ \mbox{and for any}~~~l_0<3.
\label{3.10gghhjukklllkkllloffghhjjoppuloollgghhhyhh}
\end{equation}
Now, due to
Lemma  \ref{lemma45xy1222232} 
and the second equation of \dref{1.1} and using the H\"{o}lder inequality, we have
\begin{equation}\label{3333cz2.5kke34567789999001214114114rrggjjkk}
\begin{array}{rl}
&\disp{{A}_1\int_{s_0}^t
e^{-({p}+1)(t-s)}\int_\Omega |\Delta c_\varepsilon|^{ {p}+1} ds}
\\
=&\disp{{A}_1e^{-({p}+1)t}\int_{s_0}^t
e^{({p}+1)s}\int_\Omega |\Delta c_\varepsilon|^{ {p}+1} ds}\\
\leq&\disp{2^{{p}+1}{A}_1e^{-({p}+1)t}C_ {{p}+1 }\left[\int_{s_0}^t
\int_\Omega e^{({p}+1)s}(|u_\varepsilon \cdot\nabla c_\varepsilon|^ {{p}+1 }+n^ {{p}+1 }_\varepsilon) ds+e^{({p}+1)s_0}\|c_\varepsilon(s_0,t)\|^ {{p}+1 }_{W^{2,  {{p}+1 }}}\right]}\\
\leq&\disp{2^{{p}+1}{A}_1e^{-({p}+1)t}C_ {{p}+1 }\int_{s_0}^t
 e^{({p}+1)s}(\|u_\varepsilon \|_{L^{\theta(p+1)}(\Omega)}^{p+1}\|\nabla c_\varepsilon\|_{L^{\theta'(p+1)}(\Omega)}^{p+1}+n^ {{p}+1 }_\varepsilon) ds+C_7}\\
\end{array}
\end{equation}
for all $t\in(s_0, T_{max})$,
where $\theta=\frac{8}{7},\theta'=\frac{\theta}{\theta-1}=8$,  $C_7={A}_1e^{-({p}+1)t}C_{ {p}+1}2^{{p}+1}e^{({p}+1)s_0}\|c_\varepsilon(s_0,t)\|^ {{p}+1 }_{W^{2,  {{p}+1 }}}.$
Next, an application of the Gagliardo--Nirenberg inequality and \dref{czfvgb2.5ghhjuyuccvviihjj}
infers that
\begin{equation}\label{3333cz2.5kkett677734567789999001214114114rrggjjkk}
\begin{array}{rl}
&\disp{\|\nabla c_\varepsilon\|_{L^{\theta'(p+1)}(\Omega)}^{p+1}}
\\
\leq&\disp{C_8\|\Delta c_\varepsilon\|_{L^{(p+1)}(\Omega)}^{a(p+1)}\| c_\varepsilon\|_{L^{2}(\Omega)}^{(1-a)(p+1)}+C_8\| c_\varepsilon\|_{L^{2}(\Omega)}^{p+1}}\\
\leq&\disp{C_9\|\Delta c_\varepsilon\|_{L^{(p+1)}(\Omega)}^{a(p+1)}+C_9}\\
\end{array}
\end{equation}
with some constants $C_8>0 $ and $C_9  > 0$, where
$$a=\frac{\frac{5}{6}-\frac{1}{\theta'(p+1)}}{\frac{7}{6}-\frac{1}{p+1}}\in(0,1).$$
We derive from the Young inequality that for any $\delta\in(0,1)$, 
\begin{equation}\label{3333cz2.5kkett677734gghhh567789999001214114114rrggjjkk}\begin{array}{rl}
&\disp{\|u_\varepsilon\|_{L^{\theta(p+1)}(\Omega)}^{p+1}\|\nabla c_\varepsilon\|_{L^{\theta'(p+1)}(\Omega)}^{p+1}}\\
\leq&\disp{C_9\|\Delta c_\varepsilon\|_{L^{(p+1)}(\Omega)}^{a(p+1)}\|u_\varepsilon\|_{L^{\theta (p+1)}(\Omega)}^{p+1}+C_9\|u_\varepsilon\|_{L^{\theta (p+1)}(\Omega)}^{p+1}}\\
\leq&\disp{\delta\|\Delta c_\varepsilon\|_{L^{(p+1)}(\Omega)}^{p+1}+C_{10}\|u_\varepsilon\|_{L^{\theta (p+1)}(\Omega)}^{\frac{p+1}{1-a}}+
C_{9}\|u_\varepsilon\|_{L^{\theta (p+1)}(\Omega)}^{p+1},}\\
\end{array}
\end{equation}
where $C_{10}=(1-a)\left(\delta\times\frac{1}{a}\right)^{-\frac{a}{1-a}}C_9^{\frac{1}{1-a}}.$

Substituting \dref{3333cz2.5kkett677734gghhh567789999001214114114rrggjjkk} into  \dref{3333cz2.5kke34567789999001214114114rrggjjkk} yields that
\begin{equation}\label{3333cz2ddfgggg.5kke345677ddfff89999001214114114rrggjjkk}
\begin{array}{rl}
&\disp{{A}_1\int_{s_0}^t
e^{-({p}+1)(t-s)}\int_\Omega |\Delta c_\varepsilon|^{ {p}+1} ds}
\\
\leq&\disp{2^{{p}+1}{A}_1e^{-({p}+1)t}C_ {{p}+1 }\delta\int_{s_0}^t
 e^{({p}+1)s}\|\Delta c_\varepsilon\|_{L^{(p+1)}(\Omega)}^{p+1} ds}\\
&+\disp{2^{{p}+1}{A}_1e^{-({p}+1)t}C_ {{p}+1 }\int_{s_0}^t
 e^{({p}+1)s}[C_{10}\|u_\varepsilon\|_{L^{\theta (p+1)}(\Omega)}^{\frac{p+1}{1-a}}+
C_{9}\|u_\varepsilon\|_{L^{\theta (p+1)}(\Omega)}^{p+1}] ds}\\
&\disp{+2^{{p}+1}{A}_1e^{-({p}+1)t}C_ {{p}+1 }\int_{s_0}^t
 e^{({p}+1)s}n^ {{p}+1 }_\varepsilon  ds+C_7}\\
\end{array}
\end{equation}
for all $t\in(s_0, T_{max})$.
Therefore, choosing $\delta=\frac{1}{2}\frac{1}{2^{{p}+1}C_ {{p}+1 }}$ yields to
\begin{equation}\label{3333cz2.5kke345677ddfff89999001214114114rrggjjkk}
\begin{array}{rl}
&\disp{{A}_1\int_{s_0}^t
e^{-({p}+1)(t-s)}\int_\Omega |\Delta c_\varepsilon|^{ {p}+1} ds}
\\
\leq&\disp{2^{{p}+2}{A}_1e^{-({p}+1)t}C_ {{p}+1 }C_{10}\int_{s_0}^t
 e^{({p}+1)s}\|u_\varepsilon\|_{L^{\theta (p+1)}(\Omega)}^{\frac{p+1}{1-a}}ds}\\
&+\disp{2^{{p}+2}{A}_1e^{-({p}+1)t}C_ {{p}+1 }C_{9}\int_{s_0}^t
\|u_\varepsilon\|_{L^{\theta (p+1)}(\Omega)}^{p+1} ds}\\
&\disp{+2^{{p}+2}{A}_1e^{-({p}+1)t}C_ {{p}+1 }\int_{s_0}^t
 e^{({p}+1)s}n^ {{p}+1 }_\varepsilon  ds+2C_7.}\\
\end{array}
\end{equation}
On the other hand, by Lemma \ref{lemma45630223116}, we may choose $\frac{1737}{582}<\tilde{l}_0<3$  such that
\begin{equation}\label{3333llllllcz2.5kke345677ddfff89999001214114114rrggjjkk}\frac{\frac{5}{6}-\frac{1}{\theta'(p+1)}}{\frac{7}{6}-\frac{1}{p+1}}+\frac{\frac{1}{\tilde{l}_0}-\frac{1}{\theta (p+1)}}{\frac{1}{\tilde{l}_0}+\frac{2}{3}-\frac{1}{p+1}}<1.
\end{equation}
Therefore,  it follows from the Gagliardo--Nirenberg inequality, \dref{3.10gghhjukklllkkllloffghhjjoppuloollgghhhyhh} and the Young
inequality that there exist  constants $C_{11} = C_{11}(p)> 0$ and $C_{12} = C_{12}(p)> 0$ such that
\begin{equation}\label{3333cz2.5kke345677ddff89001214114114rrggjjkk}
\begin{array}{rl}\|u_\varepsilon\|_{L^{\theta (p+1)}(\Omega)}^{\frac{p+1}{1-a}}\leq& \| Au_\varepsilon\|_{L^{p+1}(\Omega)}^{\frac{p+1}{1-a}\tilde{a}}\| u_\varepsilon\|_{L^{\tilde{l}_0}(\Omega)}^{\frac{p+1}{1-a}(1-\tilde{a})}\\
\leq& \| Au_\varepsilon\|_{L^{p+1}(\Omega)}^{\frac{p+1}{1-a}\tilde{a}}C_{11}\\
\leq& \| Au_\varepsilon\|_{L^{p+1}(\Omega)}^{p+1}+C_{12}\\
\end{array}
\end{equation}
with 
$$\tilde{a}=\frac{\frac{1}{\tilde{l}_0}-\frac{1}{\theta (p+1)}}{\frac{1}{\tilde{l}_0}+\frac{2}{3}-\frac{1}{p+1}}\in(0,1).$$
Here we have use the fact that
$\frac{p+1}{1-a}\tilde{a}=
(p+1)\frac{\frac{7}{6}-\frac{1}{p+1}}{\frac{1}{3}-\frac{1}{\theta (p+1)}}\frac{\frac{1}{\tilde{l}_0}-\frac{1}{\theta (p+1)}}
{\frac{1}{\tilde{l}_0}+\frac{2}{3}-\frac{1}{p+1}}<p+1$ by \dref{3333llllllcz2.5kke345677ddfff89999001214114114rrggjjkk}.
In light  of $\frac{1}{1-a}>1$,
similarly, we derive that
\begin{equation}\label{3333czffgyyu2.5kke345677ddff89001214114114rrggjjkk}
\begin{array}{rl}\|u_\varepsilon\|_{L^{\theta (p+1)}(\Omega)}^{p+1}\leq&  \| Au_\varepsilon\|_{L^{p+1}(\Omega)}^{p+1}+C_{13}.\\
\end{array}
\end{equation}
Then along with  \dref{3333cz2.5kke345677ddfff89999001214114114rrggjjkk}, \dref{3333cz2.5kke345677ddff89001214114114rrggjjkk} and \dref{3333czffgyyu2.5kke345677ddff89001214114114rrggjjkk}, we have
\begin{equation}\label{3333cz2.5kke345677ddffdfrrtyhhhjjf89999001214114114rrggjjkk}
\begin{array}{rl}
&\disp{{A}_1\int_{s_0}^t
e^{-({p}+1)(t-s)}\int_\Omega |\Delta c_\varepsilon|^{ {p}+1} ds}
\\
\leq&\disp{2^{{p}+2}{A}_1e^{-({p}+1)t}C_ {{p}+1 }C_{10}\int_{s_0}^t
 e^{({p}+1)s}[ \| Au_\varepsilon\|_{L^{p+1}(\Omega)}^{p+1}+C_{12}] ds}\\
&+\disp{2^{{p}+2}{A}_1e^{-({p}+1)t}C_ {{p}+1 }C_{9}\int_{s_0}^t
 e^{({p}+1)s}[ \| Au_\varepsilon\|_{L^{p+1}(\Omega)}^{p+1}+C_{13})] ds}\\
&\disp{+2^{{p}+2}{A}_1e^{-({p}+1)t}C_ {{p}+1 }\int_{s_0}^t
\int_\Omega e^{({p}+1)s}n^ {{p}+1 }_\varepsilon  ds+2C_7}\\
\leq&\disp{2^{{p}+2}{A}_1e^{-({p}+1)t}C_ {{p}+1 } [C_{10}+ C_{9}]\int_{s_0}^t
 e^{({p}+1)s}\| Au_\varepsilon\|_{L^{p+1}(\Omega)}^{p+1} ds}\\
&\disp{+2^{{p}+2}{A}_1e^{-({p}+1)t}C_ {{p}+1 }\int_{s_0}^t
\int_\Omega e^{({p}+1)s}n^ {{p}+1 }_\varepsilon  ds+C_{14}},\\
\end{array}
\end{equation}
where $C_{14}=2^{{p}+2}{A}_1e^{-({p}+1)t}C_ {{p}+1 }(C_{10}C_{12}+C_{9}C_{13})+2C_7.$
Putting $\tilde{u}_\varepsilon(\cdot, s) := e^su_\varepsilon(\cdot, s), s\in (s_0, t)$,
we obtain from the third  equation in \dref{1.1} that
\begin{equation}\label{33ffgh33cz2.5kke345677ddffdfrrtyhhhjjf89999001214114114rrggjjkk}\tilde{u}_{\varepsilon,s}=\Delta \tilde{u}_\varepsilon+\tilde{u}_\varepsilon+e^sn_\varepsilon\nabla \phi+e^s\nabla P_\varepsilon, \end{equation}
which derives
\begin{equation}\label{33ffgh33cz2.5kke345677ddffdfrrtyhhhjjf89999001214114114rrggjjkk}\tilde{u}_{\varepsilon,s}+ A\tilde{u}_\varepsilon=\mathcal{P}(\tilde{u}_\varepsilon+e^sn_\varepsilon\nabla \phi+e^s\nabla P_\varepsilon), \end{equation}
where $\mathcal{P}$ denotes the Helmholtz projection mapping $L^2(\Omega)$ onto its subspace $L^2_{\sigma}(\Omega)$ of all
solenoidal vector field.
Thus
by $p<2$ and \dref{3.10gghhjukklllkkllloffghhjjoppuloollgghhhyhh},
we derive from Lemma \ref{lemma45xy1222232} (see also Theorem 2.7 of \cite{Giga1215}) that there exist positive
constants $C_{15},C_{16},C_{17}$ and $C_{18}$
such that
\begin{equation}
\begin{array}{rl}
&\disp{\int_{s_0}^te^{(p+1)  s}\|A u_\varepsilon(\cdot,t)\|^{p+1}_{L^{p+1}(\Omega)}ds}\\
\leq &\disp{C_{15}\left(\int_{s_0}^te^{(p+1)s}
(\|u_\varepsilon(\cdot,s)\|^{p+1}_{L^{p+1}(\Omega)}+\|n_\varepsilon(\cdot,s)\|^{p+1}_{L^{p+1}(\Omega)})ds+e^{(p+1)  t}+1\right)}\\
\leq &\disp{C_{16}\left(\int_{s_0}^te^{(p+1)s}
(\|u_\varepsilon(\cdot,s)\|^{l_0}_{L^{p+1}(\Omega)}|\Omega|^{\frac{\tilde{l}_0-p-1}{\tilde{l}_0}}+\|n_\varepsilon(\cdot,s)\|^{p+1}_{L^{p+1}(\Omega)})ds+e^{(p+1)  t}+1\right)}\\
\leq &\disp{C_{17}\int_{s_0}^te^{(p+1)s}\|n_\varepsilon(\cdot,s)\|^{p+1}_{L^{p+1}(\Omega)}ds+(1+C_{18})e^{(p+1)  t}.}\\
\end{array}
\label{cz2.5bbhjjkkkiiooov114}
\end{equation}
By virtue of \dref{cz2.5bbhjjkkkiiooov114} and \dref{3333cz2.5kke345677ddffdfrrtyhhhjjf89999001214114114rrggjjkk}, we can see
\begin{equation}\label{3333cz2.5kke345677ddfddffgghhhjjkkffffdfrrtyhhhjjf89999001214114114rrggjjkk}
\begin{array}{rl}
&\disp{{A}_1\int_{s_0}^t
e^{-({p}+1)(t-s)}\int_\Omega |\Delta c_\varepsilon|^{ {p}+1} ds}
\\
\leq&\disp{2^{{p}+2}{A}_1e^{-({p}+1)t}C_ {{p}+1 }  [C_{10}+ C_{9}]\left(C_{17}\int_{s_0}^te^{(p+1)s}
\|n_\varepsilon(\cdot,s)\|^{p+1}_{L^{p+1}(\Omega)}ds+(1+C_{18})e^{(p+1)  t}\right)}\\
&\disp{+2^{{p}+2}{A}_1e^{-({p}+1)t}C_ {{p}+1 }\int_{s_0}^t
\int_\Omega e^{({p}+1)s}n^ {{p}+1 }_\varepsilon  ds+C_{14}}\\
\leq&\disp{  C_{19}\int_{s_0}^te^{(p+1)s}
\|n_\varepsilon(\cdot,s)\|^{p+1}_{L^{p+1}(\Omega)}ds+C_{20}},\\
\end{array}
\end{equation}
where $C_{19}=2^{{p}+2}{A}_1e^{-({p}+1)t}C_ {{p}+1 } [C_{10}+ C_{9}]C_{17}+2^{{p}+2}{A}_1e^{-({p}+1)t}C_ {{p}+1 }$ and  $$C_{20}:=C_{20}(p)=2^{{p}+2}{A}_1e^{-({p}+1)t}C_ {{p}+1 }  [C_{10}+ C_{9}](1+C_{18})e^{(p+1)  t}+C_{14}.$$
Collecting \dref{3333cz2.5kk1214114114rrgg} and  \dref{3333cz2.5kke345677ddfddffgghhhjjkkffffdfrrtyhhhjjf89999001214114114rrggjjkk}, applying Lemma \ref{lemma45630223116} and the Young inequality, we derive that
\begin{equation}
\begin{array}{rl}
&\disp{\frac{1}{{p}}\|n_\varepsilon(t) +\varepsilon\|^{{{p}}}_{L^{{p}}(\Omega)}}
\\
\leq&\disp{(2\varepsilon_1-\frac{4m(p-1)}{(m+p-1)^2}\frac{1}{\lambda_1}) \int_{s_0}^t
e^{-({p}+1)(t-s)}\int_\Omega (n_\varepsilon+\varepsilon)^{{{p}+m-\frac{1}{3}}} ds}\\
&\disp{+C_{19} \int_{s_0}^t
e^{-({p}+1)(t-s)}\int_\Omega n^{{p}+1}_\varepsilon ds+C_{21}}\\
\leq&\disp{(2\varepsilon_1-\frac{4m(p-1)}{(m+p-1)^2}\frac{1}{\lambda_1}) \int_{s_0}^t
e^{-({p}+1)(t-s)}\int_\Omega (n_\varepsilon+\varepsilon)^{{{p}+m-\frac{1}{3}}} ds}\\
&\disp{+C_{19} \int_{s_0}^t
e^{-({p}+1)(t-s)}\int_\Omega (n_\varepsilon+\varepsilon)^{{p}+1} ds+C_{21}}\\
\leq&\disp{(3\varepsilon_1-\frac{4m(p-1)}{(m+p-1)^2}\frac{1}{\lambda_1}) \int_{s_0}^t
e^{-({p}+1)(t-s)}\int_\Omega (n_\varepsilon+\varepsilon)^{{{p}+m-\frac{1}{3}}} ds+C_{22}}\\
\end{array}
\label{3333cz2.5kk121fttyuiii4114114rrgg}
\end{equation}
with $C_{21}=C_{20}+C_4(\varepsilon_1,{p})$ and $C_{22}=\frac{m-\frac{4}{3}}{{p}+m-\frac{1}{3}}\left(\varepsilon_1\frac{{p}+m-\frac{1}{3}}{{p}+1}\right)^{-\frac{p+1}{m-\frac{4}{3}} }
\left(C_{19}\right)^{\frac{{p}+m-\frac{1}{3}}{m-\frac{4}{3}} }+C_{21}$
  Thus, choosing $\varepsilon_1$ small enough
   (e.g. $\varepsilon_1<\frac{(p-1)}{(m+p-1)^2}\frac{1}{\lambda_1}$)
    in \dref{3333cz2.5kk121fttyuiii4114114rrgg}, using  \dref{eqx45xx1ddfgggg2112}, the H\"{o}lder inequality and $\varepsilon<1$, we derive that there exits a positive constant $p_0>\frac{3}{2}$ such that 
  \begin{equation}
\int_{\Omega}n^{p_0}_\varepsilon(x,t)dx\leq C_{23} ~~~\mbox{for all}~~ t\in(0,T_{max}).
\label{334444zjscz2.5297dfggggx96302222114}
\end{equation}
The proof of Lemma \ref{lemma45566645630223} is completed.
\end{proof}

Underlying the estimates established above, we can derive the following higher integrability properties by applying arguments which are essentially standard
in the analysis of the heat as well as the Stokes equations and a Moser-type iteration.
%
\begin{lemma}\label{lemma45630hhuujj}
Let $m> \frac{4}{3}$ and $\gamma$ be as in \dref{ccvvx1.731426677gg}.  Then one can find a positive constant $C$ independent of $\varepsilon$
 such that 
\begin{equation}
\|n_\varepsilon(\cdot,t)\|_{L^\infty(\Omega)}  \leq C ~~\mbox{for all}~~ t\in(0,T_{max})
\label{zjscz2.5297x9630111kk}
\end{equation}
and
\begin{equation}
\|c_\varepsilon(\cdot,t)\|_{W^{1,\infty}(\Omega)}  \leq C ~~\mbox{for all}~~ t\in(0,T_{max})
\label{zjscz2.5297x9630111kkhh}
\end{equation}
as well as
\begin{equation}
\|u_\varepsilon(\cdot,t)\|_{L^{\infty}(\Omega)}  \leq C ~~\mbox{for all}~~ t\in(0,T_{max}).
\label{zjscz2.5297x9630111kkhhffrr}
\end{equation}
Moreover, we also have
\begin{equation}
\|A^\gamma u_\varepsilon(\cdot,t)\|_{L^{2}(\Omega)}  \leq C ~~\mbox{for all}~~ t\in(0,T_{max}).
\label{zjscz2.5297x9630111kkhhffrreerr}
\end{equation}
\end{lemma}
\begin{proof}
In what follows, let $C, C_i$ denote some different constants, which are independent of $\varepsilon$, and if no special explanation, they depend at most on $\Omega, \phi, m, n_0, c_0$ and
$u_0$.

{\bf Step 1. The boundedness of $\|{c}_\varepsilon(\cdot,t)\|^{{{4}}}_{L^{{4}}(\Omega)}$ for all $t\in (0, T_{max})$}

Firstly,  taking  ${c^{3}_\varepsilon}$ as the test function for the second  equation of \dref{1.1} and using $\nabla\cdot u_\varepsilon=0$, the H\"{o}lder inequality  and \dref{334444zjscz2.5297x96302222114} yields  that
\begin{equation}
\begin{array}{rl}
&\disp\frac{1}{4}\disp\frac{d}{dt}\|{c}_\varepsilon\|^{{{4}}}_{L^{{4}}(\Omega)}+3
\int_{\Omega} {c^2_\varepsilon}|\nabla c_\varepsilon|^2+ \int_{\Omega} c^{4}_\varepsilon\\
=&\disp{\int_{\Omega} n_\varepsilon c^3_\varepsilon}\\
\leq&\disp{\left(\int_{\Omega}n^{\frac{3}{2}}_\varepsilon\right)^{\frac{2}{3}}\left(\int_{\Omega}c^{9}_\varepsilon\right)^{\frac{1}{3}}}\\
\leq&\disp{C_1\left(\int_{\Omega}c^{9}_\varepsilon\right)^{\frac{1}{3}}~~~\mbox{for all}~~t\in (0, T_{max}).}\\
\end{array}
\label{hhxxcdfssxxdccffgghvvjjcz2.5}
\end{equation}
Now, due to \dref{czfvgb2.5ghhjuyuccvviihjj}, in light of the Gagliardo--Nirenberg inequality and the Young inequality, we derive that
\begin{equation}
\begin{array}{rl}
\disp\left(\int_{\Omega}c^{9}_\varepsilon\right)^{\frac{1}{3}} =&\disp{\| { c^{2}_\varepsilon}\|^{{\frac{3}{2}}}_{L^{\frac{9}{2}}(\Omega)}}\\
\leq&\disp{C_{2}\left(\| \nabla{ c^{2}_\varepsilon}\|^{\frac{7}{5}}_{L^{2}(\Omega)}\|{ c^{2}_\varepsilon}\|^{{\frac{1}{10}}}_{L^{1}(\Omega)}+
\|{ c^{2}_\varepsilon}\|_{L^{1}(\Omega)}^{\frac{3}{2}}\right)}\\
\leq&\disp{C_{3}(\| \nabla{ c^{2}_\varepsilon}\|^{\frac{7}{5}}_{L^{2}(\Omega)}
+1)}\\
\leq&\disp{\frac{1}{4}\| \nabla{ c^{2}_\varepsilon}\|^{2}_{L^{2}(\Omega)}
+C_4~~~\mbox{for all}~~t\in (0, T_{max}).}\\
\end{array}
\label{ddffbnmbnddfgffggjjkkuuiicz2ggghddfvgbhh.htt678ddfghhhyuiihjj}
\end{equation}
Collecting  \dref{ddffbnmbnddfgffggjjkkuuiicz2ggghddfvgbhh.htt678ddfghhhyuiihjj} into \dref{hhxxcdfssxxdccffgghvvjjcz2.5}, in view of 
an ODE comparison argument entails
\begin{equation}
\begin{array}{rl}
&\disp{\int_{\Omega}   c^{4}_\varepsilon\leq C_5~~~\mbox{for all}~~ t\in (0, T_{max}).}\\
\end{array}
\label{czfvgb2.5ghffghjuyuccvviihjj}
\end{equation}

{\bf Step 2. The boundedness of $\|A^\gamma u_{\varepsilon}(\cdot, t)\|_{L^2(\Omega)}$ and $\| u_{\varepsilon}(\cdot, t)\|_{L^{\infty}(\Omega)}$ for all $t\in (0, T_{max})$}

On the basis of the variation-of-constants formula for the projected version of the third
equation in \dref{1.1fghyuisda}, we derive that
$$u_\varepsilon(\cdot, t) = e^{-tA}u_0 +\int_0^te^{-(t-\tau)A}
\mathcal{P}(n_\varepsilon(\cdot,t)\nabla\phi)d\tau~~ \mbox{for all}~~ t\in(0,T_{max}).$$
Therefore, according to standard smoothing
properties of the Stokes semigroup we see that 
\begin{equation}
\begin{array}{rl}
\|A^\gamma u_{\varepsilon}(\cdot, t)\|_{L^2(\Omega)}\leq&\disp{\|A^\gamma
e^{-tA}u_0\|_{L^2(\Omega)} +\int_0^t\|A^\gamma e^{-(t-\tau)A}h_{\varepsilon}(\cdot,\tau)d\tau\|_{L^2(\Omega)}d\tau}\\
\leq&\disp{\|A^\gamma u_0\|_{L^2(\Omega)} +C_{6}\int_0^t(t-\tau)^{-\gamma-\frac{3}{2}(\frac{1}{p_0}-\frac{1}{2})}e^{-\lambda(t-\tau)}\|h_{\varepsilon}(\cdot,\tau)\|_{L^{p_0}(\Omega)}d\tau}\\
\leq&\disp{C_{7}~~ \mbox{for all}~~ t\in(0,T_{max}),}\\
\end{array}
\label{cz2.571hhhhh51ccvvhddfccvvhjjjkkhhggjjllll}
\end{equation}
where $\gamma\in ( \frac{3}{4}, 1), h_{\varepsilon}=\mathcal{P}(n_{\varepsilon}\nabla \phi)$ and $p_0$ is the same as Lemma \ref{lemma45566645630223}.
Here we have used the fact that
$$
\|h_{\varepsilon}(\cdot,t)\|_{L^{p_0}(\Omega)}\leq C ~~~\mbox{for all}~~ t\in(0,T_{max})
$$
as well as
$$\begin{array}{rl}\disp\int_{0}^t(t-\tau)^{-\gamma-\frac{3}{2}(\frac{1}{p_0}-\frac{1}{2})}e^{-\lambda(t-\tau)}ds
\leq&\disp{\int_{0}^{\infty}\sigma^{-\gamma-\frac{3}{2}(\frac{1}{p_0}-\frac{1}{2})} e^{-\lambda\sigma}d\sigma<+\infty.}\\
\end{array}
$$
Observe that  $D(A^\gamma)$ is continuously embedded into $L^\infty(\Omega)$ by $\gamma>\frac{3}{4},$ so that,  \dref{cz2.571hhhhh51ccvvhddfccvvhjjjkkhhggjjllll} yields to
 \begin{equation}
\begin{array}{rl}
\|u_{\varepsilon}(\cdot, t)\|_{L^\infty(\Omega)}\leq  C_{8}~~ \mbox{for all}~~ t\in(0,T_{max}).\\
\end{array}
\label{cz2.5jkkcvvvhjkfffffkhhgll}
\end{equation}

{\bf Step 3. The boundedness of $\|\nabla c_{\varepsilon}(\cdot, t)\|_{L^{\frac{7}{2}}(\Omega)}$  for all $t\in (0, T_{max})$}

An application of the variation of
constants formula for $c_\varepsilon$ leads to
\begin{equation}
\begin{array}{rl}
&\disp{\|\nabla c_\varepsilon(\cdot, t)\|_{L^{\frac{7}{2}}(\Omega)}}\\
\leq&\disp{\|\nabla e^{t(\Delta-1)} c_0\|_{L^{\frac{7}{2}}(\Omega)}+
\int_{0}^t\|\nabla e^{(t-s)(\Delta-1)}(n_\varepsilon(s)\|_{L^{\frac{7}{2}}(\Omega)}ds}\\
&\disp{+\int_{0}^t\|\nabla e^{(t-s)(\Delta-1)}\nabla \cdot(u_{\varepsilon}(s) c_{\varepsilon}(s))\|_{L^{\frac{7}{2}}(\Omega)}ds,}\\
\end{array}
\label{44444zjccfgghhhfgbhjcvvvbscz2.5297x96301ku}
\end{equation}
To estimate the terms on the right of  \dref{44444zjccfgghhhfgbhjcvvvbscz2.5297x96301ku}, we  use the $L^p$-$L^q$ estimates associated heat semigroup
   to get that
\begin{equation}
\begin{array}{rl}
\|\nabla e^{t(\Delta-1)} c_0\|_{L^{\frac{7}{2}}(\Omega)}\leq &\disp{C_{9}~~ \mbox{for all}~~ t\in(0,T_{max})}\\
\end{array}
\label{zjccffgbhjcghhhjjjvvvbscz2.5297x96301ku}
\end{equation}
as well as
\begin{equation}
\begin{array}{rl}
&\disp{\int_{0}^t\|\nabla e^{(t-s)(\Delta-1)}n_\varepsilon(s)\|_{L^{\frac{7}{2}}(\Omega)}ds}\\
\leq&\disp{C_{10}\int_{0}^t[1+(t-s)^{-\frac{1}{2}-\frac{3}{2}(\frac{1}{p_0}-\frac{2}{7})}] e^{-(t-s)}\|n_\varepsilon(s)\|_{L^{p_0}(\Omega)}ds}\\
\leq&\disp{C_{11}~~ \mbox{for all}~~ t\in(0,T_{max})}\\
\end{array}
\label{zjccffgbhjcvvvbscz2.5297x96301ku}
\end{equation}
and
\begin{equation}
\begin{array}{rl}
&\disp{\int_{0}^t\|\nabla e^{(t-s)(\Delta-1)}\nabla \cdot(u_\varepsilon(s) c_\varepsilon(s))\|_{L^{\frac{7}{2}}(\Omega)}ds}\\
\leq&\disp{C_{12}\int_{0}^t\|(-\Delta+1)^\iota e^{(t-s)(\Delta-1)}\nabla \cdot(u_\varepsilon(s) c_\varepsilon(s))\|_{L^{4}(\Omega)}ds}\\
\leq&\disp{C_{13}\int_{0}^t(t-s)^{-\iota-\frac{1}{2}-\tilde{\kappa}} e^{-\lambda(t-s)}\|u_\varepsilon(s) c_\varepsilon(s)\|_{L^{4}(\Omega)}ds}\\
\leq&\disp{C_{14}\int_{0}^t(t-s)^{-\iota-\frac{1}{2}-\tilde{\kappa}} e^{-\lambda(t-s)}\|u_\varepsilon(s)\|_{L^{\infty}(\Omega)}\| c_\varepsilon(s)\|_{L^{4}(\Omega)}ds}\\
\leq&\disp{C_{15}~~ \mbox{for all}~~ t\in(0,T_{max}).}\\
\end{array}
\label{zjccffgbhjcvdgghhhhdfgghhvvbscz2.5297x96301ku}
\end{equation}
where $\iota=\frac{13}{28},\tilde{\kappa}=\frac{1}{56}$.
Here we have use the fact that   \dref{ccvvx1.731426677gg}, \dref{334444zjscz2.5297x96302222114} as well as  $\frac{1}{2}+\frac{3}{2}(\frac{1}{4}-\frac{2}{7})<\iota$ and $\min\{-\iota-\frac{1}{2}-\tilde{\kappa},
-\frac{1}{2}-\frac{3}{2}(\frac{1}{p_0}-\frac{2}{7})\}>-1$.
Combined with \dref{44444zjccfgghhhfgbhjcvvvbscz2.5297x96301ku}--\dref{zjccffgbhjcvdgghhhhdfgghhvvbscz2.5297x96301ku}, we derive that
\begin{equation}
\begin{array}{rl}
\|\nabla c_\varepsilon(\cdot, t)\|_{L^{\frac{7}{2}}(\Omega)}\leq  C_{16}~~ \mbox{for all}~~ t\in(0,T_{max}).\\
\end{array}
\label{cz2.5jkkcvddffgggghhdfffjjkvvhjkfffffkhhgll}
\end{equation}

{\bf Step 4. The boundedness of $\|n_{\varepsilon}(\cdot, t)\|_{L^{p}(\Omega)}$  for all $p>2+m$ and $t\in (0, T_{max})$}

Taking ${(n_{\varepsilon}+\varepsilon)^{p-1}}$ as the test function for the first equation of
\dref{1.1fghyuisda}
 and combining with the second equation, using $\nabla\cdot u_\varepsilon=0$ and the Young inequality, in view of the H\"{o}lder inequality and \dref{cz2.5jkkcvddffgggghhdfffjjkvvhjkfffffkhhgll}, we obtain
\begin{equation}
\begin{array}{rl}
&\disp{\frac{1}{{p}}\frac{d}{dt}\|n_{\varepsilon}+\varepsilon\|^{{{p}}}_{L^{{p}}(\Omega)}+
\frac{m(p-1)}{2}\int_{\Omega}{(n_{\varepsilon}+\varepsilon)^{m+p-3}} |{\nabla} {n}_{\varepsilon}|^2 }
\\
\leq&\disp{\frac{(p-1)}{2m}\int_\Omega (n_{\varepsilon}+\varepsilon)^{p+1-m}|\nabla c_{\varepsilon}|^2}\\
\leq&\disp{\frac{(p-1)}{2m}\left(\int_\Omega (n_{\varepsilon}+\varepsilon)^{3(p+1-m)}\right)^{\frac{1}{3}}\left(|\nabla c_{\varepsilon}|^3\right)^{\frac{2}{3}}}\\
\leq&\disp{C_{17}\left(\int_\Omega (n_{\varepsilon}+\varepsilon)^{3(p+1-m)}\right)^{\frac{1}{3}}~~ \mbox{for all}~~ t\in(0,T_{max}).}\\
\end{array}
\label{cz2.ddffghhjj5}
\end{equation}
On the the hand, in view of $m>\frac{4}{3}>1$ and $p>2+m,$ and hence, due to  the Gagliardo--Nirenberg inequality and the Young inequality, we derive that
\begin{equation}
\begin{array}{rl}
&\disp{C_{18}\|  (n_{\varepsilon}+\varepsilon)^{\frac{p+m-1}{2}}\|
^{\frac{2(p+1-m)}{p+m-1}}_{L^{\frac{6(p+1-m)}{p+m-1}}(\Omega)}}
\\
\leq&\disp{C_{19}(\|\nabla   (n_{\varepsilon}+\varepsilon)^{\frac{p+m-1}{2}}\|_{L^2(\Omega)}^{\frac{2(3p+2-3m)}{3p+3m-4}}\|  (n_{\varepsilon}+\varepsilon)^{\frac{p+m-1}{2}}\|_{L^\frac{2}{p+m-1}(\Omega)}^{\frac{2(p+1-m)}{p+m-1}-\frac{2(3p+2-3m)}{3p+3m-4}}+\|  (n_{\varepsilon}+\varepsilon)^{\frac{p+m-1}{2}}\|_{L^\frac{2}{p+m-1}(\Omega)}^{\frac{2(p+1-m)}{p+m-1}})}\\
\leq&\disp{C_{20}(\|\nabla   (n_{\varepsilon}+\varepsilon)^{\frac{p+m-1}{2}}\|_{L^2(\Omega)}^{\frac{2(p+1-m)}{p+m-1}}+1)}\\
\leq&\disp{\frac{m(p-1)}{4}\int_{\Omega}{(n_{\varepsilon}+\varepsilon)^{m+p-3}} |{\nabla} {n}_{\varepsilon}|^2 +C_{21}~~
\mbox{for all}~~ t\in(0,T_{max}),}\\
\end{array}
\label{cz2.563022222ikddffgopl2sdfg44}
\end{equation}
which together with \dref{cz2.ddffghhjj5} and an ODE comparison argument entails  that
\begin{equation}
\begin{array}{rl}
\|n_{\varepsilon}(\cdot, t)\|_{L^p(\Omega)}\leq  C_{22}~~ \mbox{for all}~~ t\in(0,T_{max})~~~\mbox{and}~~~p>2+m.\\
\end{array}
\label{cz2.5jkkcvvvhjkfffffkhhgll}
\end{equation}

{\bf Step 5. The boundedness of $\|c_{\varepsilon}(\cdot, t)\|_{W^{1,\infty}(\Omega)}$  for all  $t\in (\tau, T_{max})$ with $\tau\in(0,T_{max})$}

Choosing $\theta\in(\frac{1}{2}+\frac{3}{7},1),$ 
 then the domain of the fractional power $D((-\Delta + 1)^\theta)\hookrightarrow W^{1,\infty}(\Omega)$. Hence, again using the $L^p$-$L^q$ estimates associated heat semigroup,
\begin{equation}
\begin{array}{rl}
&\| c_\varepsilon(\cdot, t)\|_{W^{1,\infty}(\Omega)}\\
\leq&\disp{C_{23}\|(-\Delta+1)^\theta c_{\varepsilon}(\cdot, t)\|_{L^{\frac{7}{2}}(\Omega)}}\\
\leq&\disp{C_{24}t^{-\theta}e^{-\lambda t}\|c_0\|_{L^{\frac{7}{2}}(\Omega)}+C_{24}\int_{0}^t(t-s)^{-\theta}e^{-\lambda(t-s)}
\|(n_{\varepsilon}-u_{\varepsilon} \cdot \nabla c_{\varepsilon})(s)\|_{L^{\frac{7}{2}}(\Omega)}ds}\\
\leq&\disp{C_{25}+C_{25}\int_{0}^t(t-s)^{-\theta}e^{-\lambda(t-s)}[\|n_{\varepsilon}(s)\|_{L^{\frac{7}{2}}(\Omega)}+\|u_{\varepsilon}(s)\|_{L^\infty(\Omega)}
\|\nabla c_{\varepsilon}(s)\|_{L^{\frac{7}{2}}(\Omega)}]ds}\\
\leq&\disp{C_{26}~~ \mbox{for all}~~ t\in(\tau,T_{max})}\\
\end{array}
\label{zjccffgbhjcvvvbscz2.5297x96301ku}
\end{equation}
with $\tau\in(0,T_{max})$,
where  we have used \dref{cz2.5jkkcvvvhjkfffffkhhgll}, \dref{cz2.5jkkcvddffgggghhdfffjjkvvhjkfffffkhhgll}, \dref{cz2.5jkkcvvvhjkfffffkhhgll} as well as  the H\"{o}lder inequality and
$$\int_{0}^t(t-s)^{-\theta}e^{-\lambda(t-s)}\leq \int_{0}^{\infty}\sigma^{-\theta}e^{-\lambda\sigma}d\sigma<+\infty.$$

{\bf Step 6. The boundedness of $\|n_{\varepsilon}(\cdot, t)\|_{L^{\infty}(\Omega)}$  for all  $t\in (0, T_{max})$ for all  $t\in (\tau, T_{max})$ with $\tau\in(0,T_{max})$}

In view of \dref{zjccffgbhjcvvvbscz2.5297x96301ku} and
using the outcome of \dref{cz2.ddffghhjj5} with suitably large $p$ as a starting point, 
 which by means of
a Moser-type iteration (see  e.g. Lemma A.1 of  \cite{Tao794}) applied to the first equation of \dref{1.1fghyuisda} to get that
\begin{equation}
\begin{array}{rl}
\|n_{\varepsilon}(\cdot, t)\|_{L^{{\infty}}(\Omega)}\leq C_{27} ~~ \mbox{for all}~~~  t\in(\tau,T_{max}) \\
\end{array}
\label{cz2.5g5gghh56789hhjui78jj90099}
\end{equation}
with $\tau\in(0,T_{max})$.

{\bf Step 7. The boundedness of $\|c_{\varepsilon}(\cdot, t)\|_{W^{1,\infty}(\Omega)}$ and $\|n_{\varepsilon}(\cdot, t)\|_{L^{\infty}(\Omega)}$   for all  $t\in (0, T_{max})$}

In light of \dref{eqx45xx1ddfgggg2112}, \dref{zjccffgbhjcvvvbscz2.5297x96301ku} and \dref{cz2.5g5gghh56789hhjui78jj90099}, we conclude that
\begin{equation}
\begin{array}{rl}
\|n_{\varepsilon}(\cdot, t)\|_{L^{{\infty}}(\Omega)}\leq C_{28} ~~ \mbox{for all}~~~  t\in(0,T_{max}) \\
\end{array}
\label{cz2.5g5ddfgggggghh56789hhjui78jj90099}
\end{equation}
and
\begin{equation}
\begin{array}{rl}
\|\nabla c_{\varepsilon}(\cdot, t)\|_{L^{{\infty}}(\Omega)}\leq C_{29} ~~ \mbox{for all}~~~  t\in(0,T_{max}). \\
\end{array}
\label{cz2.5g5ddfggggggdffgghh56789hhjui78jj90099}
\end{equation}
 The proof is complete.
\end{proof}
By virtue of \dref{1.163072x} and Lemma \ref{lemma45630hhuujj}, the local-in-time solution can be extended to the global-intime
solution.
\begin{proposition}\label{lemma45630hhuujjtydfrjj} Let
$(n_\varepsilon, c_\varepsilon, u_\varepsilon, P_\varepsilon)_{\varepsilon\in(0,1)}$
 be classical solutions of \dref{1.1fghyuisda} constructed in Lemma
\ref{lemma70} on $[0, T_{max})$.  Then the solution is global on $[0,\infty)$. Moreover, one can find 
$C > 0$
independent of $\varepsilon\in(0, 1)$ such that
\begin{equation}
\|n_\varepsilon(\cdot,t)\|_{L^\infty(\Omega)}  \leq C ~~\mbox{for all}~~ t\in(0,\infty)
\label{zjscz2.5297x9630111kkuu}
\end{equation}
as well as
\begin{equation}
\|c_\varepsilon(\cdot,t)\|_{W^{1,\infty}(\Omega)}  \leq C ~~\mbox{for all}~~ t\in(0,\infty)
\label{zjscz2.5297x9630111kkhhii}
\end{equation}
and
\begin{equation}
\|u_\varepsilon(\cdot,t)\|_{W^{1,\infty}(\Omega)}  \leq C ~~\mbox{for all}~~ t\in(0,\infty).
\label{zjscz2.5297x9630111kkhhffrroo}
\end{equation}
Moreover, we also have
\begin{equation}
\|A^\gamma u_\varepsilon(\cdot,t)\|_{L^{2}(\Omega)}  \leq C ~~\mbox{for all}~~ t\in(0,\infty).
\label{zjscz2.5297x9630111kkhhffrreerrpp}
\end{equation}
\end{proposition}
\begin{lemma}\label{lemma45630hhuujjuuyy}
Let $m> \frac{4}{3}$.
Then one can find $\mu\in(0, 1)$ such that for some $C > 0$
%
%
\begin{equation}
\|c_\varepsilon(\cdot,t)\|_{C^{\mu,\frac{\mu}{2}}(\Omega\times[t,t+1])}  \leq C ~~\mbox{for all}~~ t\in(0,\infty)
\label{zjscz2.5297x9630111kkhhiioo}
\end{equation}
as well as
\begin{equation}
\|u_\varepsilon(\cdot,t)\|_{C^{\mu,\frac{\mu}{2}}(\Omega\times[t,t+1])} \leq C ~~\mbox{for all}~~ t\in(0,\infty),
\label{zjscz2.5297x9630111kkhhffrroojj}
\end{equation}
and such that for any $\tau> 0$
 there exists $C(\tau) > 0$ fulfilling
\begin{equation}
\|\nabla c_\varepsilon(\cdot,t)\|_{C^{\mu,\frac{\mu}{2}}(\Omega\times[t,t+1])} \leq C ~~\mbox{for all}~~ t\in(\tau,\infty).
\label{zjscz2.5297x9630111kkhhffrreerrpphh}
\end{equation}
\end{lemma}
\begin{proof}
Firstly, let $g_\varepsilon(x, t) := -c_\varepsilon+n_{\varepsilon}-u_{\varepsilon}\cdot\nabla c_{\varepsilon}.$
Then by Proposition \ref{lemma45630hhuujjtydfrjj}, we derive that
$ g_\varepsilon$
is bounded in $L^{\infty} (\Omega\times(0, T))$ for any  $\varepsilon\in(0,1)$,  we may invoke the standard parabolic regularity theory  to the second
equation of \dref{1.1fghyuisda} and  infer that
\dref{zjscz2.5297x9630111kkhhiioo} and \dref{zjscz2.5297x9630111kkhhffrreerrpphh} holds.
With the help of the Proposition \ref{lemma45630hhuujjtydfrjj} again, performing standard semigroup estimation techniques to the third equation of \dref{1.1fghyuisda}, we can finally get \dref{zjscz2.5297x9630111kkhhffrroojj}.
\end{proof}
%
%
%
%
%
%
%
%
%
%
%
%
%
%
%
%
%
%
%
%
\section{ Passing to the limit}
To prepare our subsequent compactness properties of
$(n_\varepsilon, c_\varepsilon, u_\varepsilon, P_\varepsilon)$ by means of the Aubin-Lions lemma (See e.g. Simon \cite{Simon}), employing almost exactly the same arguments as in the proof of Lemma 5.1 in \cite{Zhengsdsd6} (see also Lemmas 3.22--3.23 of \cite{Winkler11215}), and taking advantage of Proposition \ref{lemma45630hhuujjtydfrjj}, we conclude
the following regularity property with respect to the time variable.
%
%
\begin{lemma}\label{lemma45630hhuujjuuyytt}
Let $m> \frac{4}{3}$.
Then one can find $\varepsilon\in(0, 1)$ such that for some $C > 0$
%
%
\begin{equation}
\|\partial_tn_\varepsilon(\cdot,t)\|_{(W^{2,2}_0(\Omega))^*}  \leq C ~~\mbox{for all}~~ t\in(0,\infty).
\label{zjscz2.5297x9630111kkhhiioott}
\end{equation}
Moreover,
let $\varsigma> \max\{m, 2(m - 1 )\}$. Then
for all $T > 0$ and $\varepsilon\in(0,1)$ there exists $C(T) > 0$ such that
\begin{equation}
\int_0^T\|\partial_tn_\varepsilon^\varsigma(\cdot,t)\|_{(W^{3,2}_0(\Omega))^*}dt  \leq C(T) .
\label{zjscz2.5297x9630111kkhhiioott4}
\end{equation}
\end{lemma}

Based on above lemmas and by extracting suitable subsequences in a standard
way, we could see the solution of \dref{1.1} is indeed globally solvable. To this end, from the idea of \cite{Zhengsdsd6} (see also \cite{Winkler11215} and \cite{Liuddfffff}), we state the solution conception as follows.

\begin{definition}\label{df1}
Let $T > 0$ and $(n_0, c_0, u_0)$ fulfills
\dref{ccvvx1.731426677gg}.
Then a triple of functions $(n, c, u)$ is
called a weak solution of \dref{1.1} if the following conditions are satisfied
\begin{equation}
 \left\{\begin{array}{ll}
   n\in L_{loc}^1(\bar{\Omega}\times[0,T)),\\
    c \in L_{loc}^1([0,T); W^{1,1}(\Omega)),\\
u \in  L_{loc}^1([0,T); W^{1,1}(\Omega)),\\
 \end{array}\right.\label{dffff1.1fghyuisdakkklll}
\end{equation}
where $n\geq 0$ and $c\geq 0$ in
$\Omega\times(0, T)$ as well as $\nabla\cdot u = 0$ in the distributional sense in
 $\Omega\times(0, T)$,
moreover,
\begin{equation}\label{726291hh}
\begin{array}{rl}
 &~~ n^m~\mbox{belong to}~~ L^1_{loc}(\bar{\Omega}\times [0, \infty)),\\
  &cu,~ ~nu ~~\mbox{and}~~n\nabla c~ \mbox{belong to}~~
L^1_{loc}(\bar{\Omega}\times [0, \infty);\mathbb{R}^{3})
\end{array}
\end{equation}
and
\begin{equation}
\begin{array}{rl}\label{eqx45xx12112ccgghh}
\disp{-\int_0^{T}\int_{\Omega}n\varphi_t-\int_{\Omega}n_0\varphi(\cdot,0)  }=&\disp{
\int_0^T\int_{\Omega}n^m\Delta\varphi+\int_0^T\int_{\Omega}n\nabla c\cdot\nabla\varphi}\\
&+\disp{\int_0^T\int_{\Omega}nu\cdot\nabla\varphi}\\
\end{array}
\end{equation}
for any $\varphi\in C_0^{\infty} (\bar{\Omega}\times[0, T))$ satisfying
 $\frac{\partial\varphi}{\partial\nu}= 0$ on $\partial\Omega\times (0, T)$
  as well as
  \begin{equation}
\begin{array}{rl}\label{eqx45xx12112ccgghhjj}
\disp{-\int_0^{T}\int_{\Omega}c\varphi_t-\int_{\Omega}c_0\varphi(\cdot,0)  }=&\disp{-
\int_0^T\int_{\Omega}\nabla c\cdot\nabla\varphi-\int_0^T\int_{\Omega}c\varphi+\int_0^T\int_{\Omega}n\varphi+
\int_0^T\int_{\Omega}cu\cdot\nabla\varphi}\\
\end{array}
\end{equation}
for any $\varphi\in C_0^{\infty} (\bar{\Omega}\times[0, T))$  and
\begin{equation}
\begin{array}{rl}\label{eqx45xx12112ccgghhjjgghh}
\disp{-\int_0^{T}\int_{\Omega}u\varphi_t-\int_{\Omega}u_0\varphi(\cdot,0) }=&\disp{-
\int_0^T\int_{\Omega}\nabla u\cdot\nabla\varphi-
\int_0^T\int_{\Omega}n\nabla\phi\cdot\varphi}\\
\end{array}
\end{equation}
for any $\varphi\in C_0^{\infty} (\bar{\Omega}\times[0, T);\mathbb{R}^3)$ fulfilling
$\nabla\varphi\equiv 0$ in
 $\Omega\times(0, T)$.
 If $\Omega\times (0,\infty)\longrightarrow \mathbb{R}^5$ is a weak solution of \dref{1.1} in
 $\Omega\times(0, T)$ for all $T > 0$, then we call
$(n, c, u)$ a global weak solution of \dref{1.1}.
\end{definition}
With the uniform bounds from Proposition \ref{lemma45630hhuujjtydfrjj} and Lemma \ref{lemma45630hhuujjuuyy} we are now in the
position to obtain limit functions $n, c$ and $u$, which at least fulfill the regularity assumptions required
in Definition \ref{df1}.
\begin{lemma}\label{lemma45630223}
Assume that   $m> \frac{4}{3}$.
 Then there exists $(\varepsilon_j)_{j\in \mathbb{N}}\subset (0, 1)$ such that $\varepsilon_j\rightarrow 0$ as $j\rightarrow\infty$ and that
\begin{equation} n_\varepsilon\rightarrow n ~~\mbox{a.e.}~~ \mbox{in}~~ \Omega\times (0,\infty),
\label{zjscz2.5297x9630222222}
\end{equation}
\begin{equation}
 n_\varepsilon\rightharpoonup n ~~\mbox{weakly star in}~~ L^\infty(\Omega\times(0,\infty)),
 \label{zjscz2.5297x9630222222ee}
\end{equation}
\begin{equation}
n_\varepsilon\rightarrow n ~~\mbox{in}~~ C^0_{loc}([0,\infty); (W^{2,2}_0 (\Omega))^*),
\label{zjscz2.5297x96302222tt}
\end{equation}
\begin{equation}
c_\varepsilon\rightarrow c ~~\mbox{in}~~ C^0_{loc}(\bar{\Omega}\times[0,\infty)),
 \label{zjscz2.5297x96302222tt3}
\end{equation}
\begin{equation}
\nabla c_\varepsilon\rightarrow \nabla c ~~\mbox{in}~~ C^0_{loc}(\bar{\Omega}\times[0,\infty)),
 \label{zjscz2.5297x96302222tt4}
\end{equation}
\begin{equation}
\nabla c_\varepsilon\rightarrow \nabla c ~~\mbox{in}~~ L^{\infty}(\Omega\times(0,\infty)),
 \label{zjscz2.5297x96302222tt4}
\end{equation}
\begin{equation}
u_\varepsilon\rightarrow u ~~\mbox{in}~~ C^0_{loc}(\bar{\Omega}\times(0,\infty)),
 \label{zjscz2.5297x96302222tt44}
\end{equation}
and
\begin{equation}
D u_\varepsilon\rightharpoonup Du ~~\mbox{weakly in}~~L^{\infty}(\Omega\times[0,\infty))
 \label{zjscz2.5297x96302222tt4455}
\end{equation}
 with some triple $(n, c, u)$ which is a global weak solution of \dref{1.1} in the sense of Definition \ref{df1}. Moreover, $n$ satisfies
 \begin{equation}
n\in C^0_{\omega-*}([0,\infty); L^\infty(\Omega)).
 \label{zjscz2.5297x96302222tt4455hyuhii}
 \end{equation}
\end{lemma}
\begin{proof}
Firstly, Proposition \ref{lemma45630hhuujjtydfrjj} warrants that for certain $n\in L^\infty(\Omega\times(0,\infty))$, \dref{zjscz2.5297x9630222222ee} is valid.
Next, the bounds featured in Proposition \ref{lemma45630hhuujjtydfrjj}, we derive from \dref{cz2.ddffghhjj5} that there exists a positive constant $C_1:=C_1(T)$ such that
\begin{equation}
\int_{0}^T\int_{\Omega}n_{\varepsilon}^{m+p-3} |\nabla n_{\varepsilon}|^2\leq C_1
\label{fvgbhzjscz2.5297x96302222tt4455hyuhii}
 \end{equation}
 for any $p>1.$
 In particular, we choose $p := 2\zeta-m+1$, where $\varsigma> \max\{m,2(m-1)\}$ is the same as Lemma \ref{lemma45630hhuujjuuyytt}.
Therefore,  \dref{fvgbhzjscz2.5297x96302222tt4455hyuhii} asserts that  for each
$T > 0,$ $(n_{\varepsilon}^\varsigma)_{\varepsilon\in(0,1)}$ is bounded in $L^2((0, T);W^{1,2}(\Omega))$, so that
combined \dref{zjscz2.5297x9630111kkhhiioott4} with the Aubin-Lions lemma (see e.g. \cite{Temamdd41215}), we derive that
 $n_{\varepsilon}^\varsigma\rightarrow z^\varsigma$ for some nonnegative measurable $z:\Omega\times(0,\Omega)\rightarrow\mathbb{R}$.
Thus, \dref{zjscz2.5297x9630222222ee} and the Egorov theorem yields to  $z=n$ necessarily, and thereby \dref{zjscz2.5297x9630222222} holds.
Finally,
employing the same arguments as in the proof of Lemma 4.1 in \cite{Winkler11215} (see also \cite{Zhengsdsd6}), taking advantage of Proposition \ref{lemma45630hhuujjtydfrjj}, we can conclude \dref{zjscz2.5297x96302222tt}--\dref{zjscz2.5297x96302222tt4455hyuhii},  where the required equicontinuity
property used in the proof, is implied by \dref{zjscz2.5297x9630111kkhhiioott}.
%
The proof of Lemma \ref{lemma45630223} is completed.
\end{proof}

{\bf Proof of Theorem \ref{theorem3}}: The conclusion in Theorem \ref{theorem3} follows from Lemma \ref{lemma45630223} and
Proposition \ref{lemma45630hhuujjtydfrjj}.

{\bf Acknowledgement}:
This work is partially supported by  the National Natural
Science Foundation of China (No. 11601215), Shandong Provincial
Science Foundation for Outstanding Youth (No. ZR2018JL005), Shandong Provincial
Natural Science Foundation of China (No. ZR2016AQ17) and the Doctor Start-up Funding of Ludong University (No. LA2016006).


\begin{thebibliography}{00}












\bibitem{Bellomo1216} N. Bellomo,  A. Belloquid,   Y. Tao, M. Winkler,  \textit{Toward a mathematical theory of
Keller--Segel models of pattern formation in biological tissues}, Math. Models Methods Appl. Sci., 25(9)(2015), 1663--1763.









\bibitem{Chaexdd12176} M. Chae, K. Kang, J. Lee, \textit{Global Existence and temporal decay in Keller--Segel models
coupled to
fluid equations}, Comm. Part. Diff. Eqns., 39(2014), 1205--1235.

 \bibitem{Cie791} T. Cie\'{s}lak,  C. Stinner, \textit{Finite-time blowup and global-in-time unbounded solutions to a parabolic--parabolic quasilinear
Keller--Segel system in higher dimensions,} J. Diff. Eqns., 252(2012), 5832--5851.

 \bibitem{Cie201712791} T. Cie\'{s}lak,  C. Stinner, \textit{New critical exponents in a fully parabolic quasilinear Keller--Segel system and applications
to volume filling models}, J. Diff. Eqns., 258(2015), 2080--2113.



\bibitem{Cie72}T. Cie\'{s}lak, M. Winkler,  \textit{Finite-time blow-up in a quasilinear system of chemotaxis,}
Nonlinearity, 21(2008), 1057--1076.









\bibitem{Duan12186} R. Duan, A. Lorz, P. A. Markowich, \textit{Global solutions to the coupled chemotaxis-
fluid
equations}, Comm. Part. Diff. Eqns., 35 (2010), 1635--1673.




\bibitem{Giga1215}  Y. Giga, \textit{Solutions for semilinear parabolic equations in $L^p$ and regularity of weak solutions of the
Navier--Stokes system}, J. Diff. Eqns., 61(1986), 186--212.

\bibitem{Guggg1215}	J. Gu, F. Meng,  \textit{Some new nonlinear Volterra-Fredholm type dynamic integralinequalities on time scales}, Applied Mathematics and Computation,  245(2014),  235--242.


\bibitem{Hieber}  M. Hieber, J. Pr\"{u}ss, \textit{Heat kernels and maximal $L^p$-$L^q$ estimate for parabolic evolution equations}, Comm. Partial Diff. Eqns., 22(1997), 1647--1669.


\bibitem{Hillen} T. Hillen,   K. Painter, \textit{A user's guide to PDE models for chemotaxis,} J. Math. Biol., 58(2009), 183--217.


\bibitem{Hillen334} T. Hillen,   K. Painter, \textit{Global existence for a parabolic chemotaxis model with prevention of overcrowding},
 Adv. Appl.
Math., 26(2001), 281--301.








\bibitem{Horstmann2710} D. Horstmann,  \textit{From $1970$ until present: the Keller--Segel model in chemotaxis and its consequences,} I.
 Jahresberichte der Deutschen Mathematiker-Vereinigung, 105(2003), 103--165.


\bibitem{Horstmann791} D. Horstmann, M. Winkler, \textit{Boundedness vs. blow-up in a chemotaxis system}, J. Diff. Eqns, 215(2005), 52--107.



\bibitem{Ishidaddssdd} S. Ishida, T. Yokota, \textit{Blow-up in finite or infinite time for quasilinear degenerate Keller-Segel systems of parabolic-parabolic type}, Discrete Contin. Dyn. Syst. Ser. B., 18(2013), 2569--2596.


\bibitem{Keller2710}E. Keller, L. Segel,  \textit{Model for chemotaxis}, J. Theor. Biol., 30(1970),  225--234.


\bibitem{Liggh00} F. Li, Q. Gao,  \textit{Blow-up of solution for a nonlinear Petrovsky type equation with memory}, Applied  Mathematic  and  Computation,
 274(2016), 383--392.


\bibitem{Liggghh793}  X. Li, Y. Wang,  Z. Xiang,  \textit{Global existence and boundedness in a 2D Keller--Segel--Stokes system with nonlinear
diffusion and rotational flux}, Commun. Math. Sci., 14(2016), 1889--1910.




\bibitem{Liu1215} J.-G. Liu, A. Lorz, \textit{A coupled chemotaxis--fluid model: global existence}, Ann. Inst. H.
Poincar\'{e} Anal. Non Lin\'{e}aire, 28(5)(2011), 643--652.

\bibitem{Liuddfffff} J. Liu, Y. Wang, \textit{Boundedness and decay property in a three-dimensional Keller-Segel-Stokes system involving
tensor-valued sensitivity with saturation}, J. Diff. Eqns., 261(2)(2016), 967--999.





\bibitem{LiuZhLiuLiuandddgddff4556} J. Liu, Y. Wang, \textit{Global weak solutions in a three-dimensional Keller-Segel-Navier-Stokes system involving a
tensor-valued sensitivity with saturation}, J. Diff. Eqns., 262(10)(2017), 5271--5305.


  \bibitem{Lorz1215}  A. Lorz, \textit{Coupled chemotaxis fluid equations}, Math. Models Methods Appl. Sci., 20(2010), 987--1004.



\bibitem{Miller7gg6} R. L. Miller, \textit{Demonstration of sperm chemotaxis in echinodermata: Asteroidea, Holothuroidea, Ophiuroidea},  J.
Exp. Zool., 234(3):383--414, 1985.





\bibitem{Painter55677}  K. Painter, T. Hillen, \textit{Volume-filling and quorum-sensing in models for chemosensitive movement}, Can. Appl. Math.
Q. 10(2002), 501--543.

\bibitem{Peng55667}Y. Peng, Z. Xiang, \textit{Global existence and boundedness in a 3D Keller--Segel--Stokes system with nonlinear
diffusion and rotational flux}, Z. Angew. Math. Phys., (2017), 68:68.






\bibitem{Simon} J. Simon, \textit{Compact sets in the space $L^{p}(O, T;B)$}, Annali di Matematica Pura ed Applicata, 146(1)(1986), 65--96.



\bibitem{Sohr} H. Sohr, \textit{The Navier--Stokes equations,
 An elementary functional analytic approach},
 Birkh\"{a}user Verlag, Basel (2001).



%
%

\bibitem{Tao794} Y. Tao, M. Winkler,  \textit{Boundedness in a quasilinear parabolic--parabolic Keller--Segel system with subcritical sensitivity}, J.
Diff. Eqns., 252(2012), 692--715.









  \bibitem{Tao41215}  Y. Tao, M. Winkler, \textit{Boundedness and decay enforced by quadratic degradation in a three-dimensional
chemotaxis--fluid system}, Z. Angew. Math. Phys., 66(2015), 2555--2573.







%
%


  \bibitem{Temamdd41215}  R. Temam,  \textit{Navier-Stokes equations. Theory and numerical analysis}, Studies in Mathematics
and its Applications, Vol. 2. North-Holland, Amsterdam, 1977.


\bibitem{Tuval1215}  I. Tuval, L. Cisneros, C. Dombrowski, et al., \textit{Bacterial swimming and oxygen transport near contact
lines}, Proc. Natl. Acad. Sci. USA, 102(2005), 2277--2282.


 \bibitem{Wangssddss21215}  Y. Wang,  \textit{Global weak solutions in a three-dimensional
Keller-Segel-Navier-Stokes system
with subcritical sensitivity},  Math. Models Methods Appl. Sci., (27)(14)(2017), 2745--2780.







\bibitem{Wangsseeess21215} 	Y. Wang, L. Liu, X. Zhang, Y. Wu, \textit{Positive solutions of  a  fractional semipositone differential system arising from the study of  HIV infection models}, Aplied Math. Comput., 258(2015), 312-324.




\bibitem{Wang21215}  Y. Wang, Z. Xiang, \textit{Global existence and boundedness in a Keller--Segel--Stokes system involving a
tensor-valued sensitivity with saturation}, J. Diff. Eqns., 259(2015), 7578--7609.


\bibitem{Wangss21215}  Y. Wang, Z. Xiang, \textit{Global existence and boundedness in a Keller-Segel-Stokes system involving a tensor-valued
sensitivity with saturation: the 3D case}, J. Differ. Eqns. 261(2016), 4944--4973.


\bibitem{Wangffgggss21215}  Z. Wang, M. Winkler, D. Wrzosek, \textit{Global regularity vs. infinite-time singularity formation in a chemotaxis model with volume-filling effect and degenerate diffusion}, SIAM J. Math. Anal., 44(2012), 3502--3525.


 \bibitem{Wiegnerdd79} M. Wiegner, \textit{The Navier-Stokes equations¡ªa neverending challenge}? Jahresber. Deutsch.
Math.-Verein., 101(1)(1999), 1--25.

    \bibitem{Winkler79} M. Winkler, \textit{Does a volume-filling effect always prevent chemotactic collapse}, Math. Methods Appl. Sci., 33(2010), 12--24.




 \bibitem{Winkler21215} M. Winkler, \textit{Boundedness in the higher-dimensional parabolic--parabolic chemotaxis system with
logistic source}, Comm.  Partial Diff. Eqns., 35(2010), 1516--1537.


\bibitem{Winkler31215}  M. Winkler, \textit{Global large-data solutions in a chemotaxis--(Navier--)Stokes system modeling cellular swimming in
fluid drops}, Comm. Partial Diff. Eqns., 37(2012), 319--351.






  \bibitem{Winkler61215}  M. Winkler, \textit{Stabilization in a two-dimensional chemotaxis--Navier--Stokes system}, Arch. Ration. Mech.
Anal., 211(2014), 455--487.









 \bibitem{Winkler11215}  M. Winkler, \textit{Boundedness and large time behavior in a three-dimensional chemotaxis--Stokes system
with nonlinear diffusion and general sensitivity},
Calculus of Variations and Partial Diff. Eqns., (54)(2015),   3789--3828.





\bibitem{Winkler51215}  M. Winkler, \textit{Global weak solutions in a three-dimensional chemotaxis--Navier--Stokes system}, Ann. Inst.
H. Poincar\'{e} Anal. Non Lin\'{e}aire, 33(5)(2016),  1329---1352.





\bibitem{Winkler72} M. Winkler, K. C. Djie, \textit{Boundedness and finite-time
collapse in a  chemotaxis system with volume-filling effect}, Nonlinear Anal. TMA.,  72(2010),  1044--1064.



%
%

\bibitem{Xu5566r793}	F. Xu, L. Liu,  \textit{On the well-posedness of the incompressible flow in porous media}, J. Nonlinear Sci. Appl., 9(12)(2016), 6371--6381.

 \bibitem{Zhangdddddff4556} Q. Zhang, Y. Li, \textit{Global weak solutions for the three-dimensional chemotaxis-Navier-Stokes system with nonlinear
diffusion}, J. Diff. Eqns., 259(8)(2015), 3730--3754.

\bibitem{Zhang12176}  Q. Zhang, X. Zheng, \textit{Global well-posedness for the two-dimensional incompressible
chemotaxis--Navier--Stokes equations}, SIAM J. Math. Anal., 46(2014), 3078--3105.


\bibitem{Zheng00} J. Zheng, \textit{Boundedness of solutions to a quasilinear parabolic--elliptic Keller--Segel system with logistic source}, J. Diff.
Eqns., 259(1)(2015), 120--140.

\bibitem{Zheng33312186} J. Zheng, \textit{Boundedness of solutions to a quasilinear parabolic--parabolic Keller--Segel system with logistic source},
J. Math. Anal. Appl.,  431(2)(2015),  867--888.




\bibitem{Zhengsdsd6} J. Zheng,   \textit{Boundedness in a three-dimensional chemotaxis--fluid system involving tensor-valued sensitivity with saturation},
J. Math. Anal. Appl., 442(1)(2016), 353--375.







\bibitem{Zhengssdefr23} J. Zheng, \textit{A note on boundedness of solutions to a higher-dimensional quasi-linear chemotaxis system with logistic source},
Zeitsc.  Angew. Mathe. Mech., 97(4)(2017),  414--421.


 \bibitem{Zhengddkkllssssssssdefr23} J. Zheng, \textit{Boundedness and global asymptotic stability of constant equilibria in a fully parabolic chemotaxis system with nonlinear logistic source},  J. Math. Anal. Appl., 450(2)(2017), 1047--1061.




 \bibitem{Zhengsddfffsdddssddddkkllssssssssdefr23} J. Zheng, \textit{Global weak solutions in a three-dimensional Keller-Segel-Navier-Stokes system with nonlinear diffusion},
J. Diff. Eqns., 263(5)(2017), 2606--2629.


  \bibitem{Zhengssdddssddddkkllssssssssdefr23} J. Zheng et. al., \textit{A new result for global existence and boundedness of solutions to a parabolic--parabolic Keller--Segel system with logistic source}, J. Math. Anal. Appl., 462(1)(2018), 1--25.



\bibitem{Zhengssdddd00} J. Zheng, \textit{An optimal result for global existence and boundedness in a three-dimensional Keller-Segel(-Navier)-Stokes system (involving a tensor-valued sensitivity with saturation}, arXiv:1806.07067.

\bibitem{Zhengsddfff00} J. Zheng, \textit{A new result for global existence and boundedness in a three-dimensional
 Keller-Segel(-Navier)-Stokes system   with  nonlinear diffusion},   Preprint.

\bibitem{Zhengssdddddfssdddd00} J. Zheng, \textit{An optimal  result  for global existence and boundedness in a
three-dimensional
 attraction-repulsion chemotaxis-Stokes system with  nonlinear diffusion}, Preprint.





%


 \bibitem{Zhengssssdefr23} J. Zheng, Y. Wang, \textit{Boundedness and decay behavior in a higher-dimensional quasilinear chemotaxis system with nonlinear logistic source}, Compu. Math. Appl., 72(10)(2016), 2604--2619.




     \bibitem{Zhengssssssdefr23} J. Zheng, Y. Wang, \textit{A note on global existence to a higher-dimensional quasilinear chemotaxis system with consumption of chemoattractant},
Discrete Contin. Dyn. Syst. Ser. B,  22(2)(2017),
669--686.



\end{thebibliography}
\end{document}